\documentclass[preprint]{elsarticle}
\usepackage{amssymb,latexsym,amsthm}
\usepackage{amsmath,amsfonts}
\usepackage{enumerate}
\usepackage{geometry}
\usepackage{url}
\usepackage[british]{babel}
\usepackage{microtype}

\newcommand{\cS}{\mathcal{S}}
\newcommand{\cQ}{\mathcal{Q}}

\newcommand{\cH}{\mathcal{H}}

\newcommand{\HH}{\mathbb{H}}
\newcommand{\PP}{\mathbb{P}}

\newcommand{\QQ}{\mathbb{Q}}
\newcommand{\cG}{\mathcal{G}}

\newcommand{\PG}{\mathrm{PG}}

\newcommand{\FF}{\mathbb{F}}

\newcommand{\bS}{\mathbb{S}}

\newcommand{\Fix}{\mathrm{Fix}}

\newcommand{\ch}{\mathrm{char}}

\newcommand{\cP}{\mathcal{P}}

\theoremstyle{plain}
\newtheorem{MainTheo}{Theorem}
\newtheorem{MainCor}[MainTheo]{Corollary}
\newtheorem{lemma}{Lemma}[section]
\newtheorem{theorem}[lemma]{Theorem}
\newtheorem{corollary}[lemma]{Corollary}

\newtheorem{setting}[lemma]{Setting}
\newtheorem{conjecture}[lemma]{Conjecture}
\theoremstyle{definition}
\newtheorem{definition}[lemma]{Definition}
\newtheorem{remark}[lemma]{Remark}

\begin{document}
\begin{frontmatter}
\title{Grassmann embeddings of polar Grassmannians}

\author[SI]{Ilaria Cardinali}
\ead{ilaria.cardinali@unisi.it}
\address[SI]{Department of Information Engineering and Mathematics, University of Siena,
Via Roma 56, I-53100, Siena, Italy}
\author[LG]{Luca Giuzzi\corref{cor1}}
\ead{luca.giuzzi@unibs.it}
\address[LG]{DICATAM - Section of Mathematics,
University of Brescia,
Via Branze 53, I-25123, Brescia, Italy}
\cortext[cor1]{Corresponding author. Tel. +39 030 3715739; Fax. +39 030 3615745}
\author[SI,r]{Antonio Pasini}
\ead{antonio.pasini@unisi.it}
\fntext[r]{(retired)}
\begin{abstract}
In this paper we compute the dimension of the Grassmann embeddings of the polar Grassmannians associated to a possibly degenerate Hermitian, alternating or quadratic form with possibly non-maximal Witt index. Moreover, in the characteristic $2$ case, when the form is quadratic and non-degenerate with bilinearization of minimal Witt index, we define a generalization of the so-called Weyl embedding (see \cite{IP13}) and prove that the Grassmann embedding is a quotient of this generalized `Weyl-like' embedding. We also estimate the dimension of the latter.
\end{abstract}
\end{frontmatter}
\section{Introduction}

Let  $V:=V(N,\FF)$ be  a vector space of dimension $1<N<\infty$ over a field $\FF$, equipped with a (possibly degenerate) sesquilinear or quadratic form $\eta$ such that $V$ is spanned by the set of the vectors that are singular for $\eta$.
Let $R := \mathrm{Rad}(\eta)$ be the radical of $\eta$ and let $n$ be the
\emph{reduced Witt index} of $\eta$, namely the Witt index of the non-degenerate form induced by $\eta$ on $V/R$. The numbers $d := \dim(R)$ and $n+d$ are the (\emph{singular}) \emph{defect} and the \emph{Witt index} of $\eta$, respectively. With respect to $\eta$, the space $V$ admits a direct sum decomposition
\begin{equation}\label{decomposizione di base}
V = (\bigoplus_{i=1}^nV_i)\oplus V_0 \oplus R
\end{equation}
where $V_1, V_2,\dots, V_n$ are mutually orthogonal hyperbolic $2$-spaces and $V_0$ is an $(N-2n-d)$-dimensional totally anisotropic subspace orthogonal to each of $V_1, V_2,\dots, V_n$. In order to avoid trivial cases we assume $n > 1$. We call the number $d_0 := \dim(V_0)$ the \emph{anisotropic defect} of $\eta$ and we denote it $\mathrm{def}_0(\eta)$, while $\mathrm{def}(\eta)$ stands for $d$.

For $1 \leq k < N$ denote by $\cG_k$  the $k$-Grassmannian of $V$, that is the point-line geometry having as points the subspaces of $V$ of dimension $k$ and as lines the sets of the form $\ell_{X,Y}:=\{Z\colon  X<Z<Y, \dim(Z) = k\}$, where $X$ and $Y$ are two subspaces of $V$ with $\dim(X)=k-1$, $\dim(Y)=k+1$
and $X<Y$. Incidence is containment.

Let $e_k:\cG_{k}\to\PG(\bigwedge^k V)$ be the Pl\"ucker (or Grassmann) embedding  of $\cG_k$, mapping the point $\langle v_1,\dots,v_k\rangle$ of $\cG_k$ to the projective point $\langle v_1\wedge\dots\wedge v_k\rangle$ of $\PG(\bigwedge^k V)$. The dimension of an embedding is defined as the vector dimension of the space spanned by its image. It is well known that $\dim(e_k)= {N\choose k}$.

With $\eta$ a form of reduced Witt index $n>1$ and singular defect $d$, for $k=1,\dots,n+d$  the \emph{polar $k$-Grassmannian} associated to $\eta$ is the point-line geometry having as points the totally $\eta$--singular $k$-dimensional subspaces of $V$. Lines are defined as follows:
\begin{enumerate}
\item For $k< n+d$, the lines are the lines $\ell_{X,Y}$ of $\cG_k$ with $Y$ totally $\eta$-singular. 
\item
  For $k=n+d$, if $\eta$ is sesquilinear then the lines are the sets as follows
  \[ \ell_{X}:=\{ Y: X\leq Y, \dim(Y)=n+d, \,Y\,\, {\text{totally $\eta$--singular}}\}  \,\,\mbox{\rm  with}\,\, \dim(X)=n+d-1\]
  and $X$ totally $\eta$-singular,
while when $\eta$ is a quadratic form they are the sets
\[ \ell_{X}:=\{ Y: X\leq Y\leq X^{\perp}, \dim(Y)=n+d, \,Y \,\, {\text{totally $\eta$--singular}}\} \,\,\mbox{\rm  with}\,\, \dim(X)=n+d-1,\]
where $X$ is totally $\eta$-singular and $X^\perp$ is the orthogonal of $X$ w.r.t. the bilinearization of $\eta$.
\end{enumerate}
Let now $\cP_k$ be the polar $k$-Grassmannian defined by $\eta$. If $k=1$ the geometry $\cP_1$ is the polar space associated to $\eta$. When $d = 0$ (namely $\eta$ is non-degenerate) and $k = n$, the polar Grassmannian $\cP_{n}$ is usually called \emph{dual polar space}.

If $k < n+d$ then $\cP_k$ is a full-subgeometry of $\cG_k$. In any case, all points of $\cP_k$ are points of $\cG_k$. So,
we can consider the restriction $\varepsilon_{k}:=e_{k}|_{\cP_{k}}$ of the Pl\"ucker embedding $e_{k}$ of $\cG_{k}$
 to $\cP_{k}$. The map $\varepsilon_{k}$ is an embedding of $\cP_{k}$ called \emph{Pl\"ucker (or Grassmann) embedding} of $\cP_{k}$.

 Note that the span $\langle\varepsilon_k(\cP_k)\rangle$ of the $\varepsilon_k$-image $\varepsilon_k(\cP_k)$ of (the point-set of) $\cP_k$
 might not
 coincide with $\PG(\bigwedge^k V)$, although the equality $\langle\varepsilon_k(\cP_k)\rangle = \mathrm{PG}(\bigwedge^kV)$
 holds in several cases, as we shall see in this paper. The \emph{dimension} $\dim(\varepsilon_k)$ of $\varepsilon_k$ is the dimension of the vector subspace of $\bigwedge^kV$ corresponding to $\langle\varepsilon_k(\cP_k)\rangle$.

When $k<n+d$ the embedding $\varepsilon_k$ is \emph{projective}, namely it maps the lines of $\cP_k$ surjectively onto lines of $\PG(\bigwedge^kV)$.
When $k=n+d$ this is not true, except when $d = 0$ and $\eta$ is an alternating form. Indeed when $\eta$ is degenerate ${\cP}_{n+d}$ also admits lines consisting of just one point. Accordingly, when $d > 0$ the geometry
${\cP}_{n+d}$ does not admit any projective embedding. If $\eta$ is non-degenerate but not alternating then $\varepsilon_n$ may map the lines of $\cP_n$ onto proper sublines of lines of $\PG(\bigwedge^kV)$
 (in which case $\varepsilon_n$ is \emph{laxly projective}) or onto conics (when $\varepsilon_n$ is a \emph{veronesean embedding} as
 defined in~\cite{IP13}, also~\cite{Pas13}) or other curves or varieties, depending on $d_0$ and the type of $\eta$.

 When $\eta$ is sesquilinear but not bilinear we can always assume that it is Hermitian. If it is bilinear then it can be either alternating or symmetric. However, if $\eta$ is symmetric we can always replace it with the quadratic form associated to it.
Thus, henceforth, we shall consider only Hermitian, alternating and quadratic forms.

\begin{definition}\label{HSQ}
If  $\eta$ is a Hermitian form, then $\cP_k$ is called \emph{Hermitian} $k$-\emph{Grassmannian}.  We denote it by ${\cH}_k(n,d_0,d;\FF)$, as to have a notation which keeps record of the reduced Witt index $n$ and the anisotropic and singular defects $d_0$ and $d$ of $\eta$.

When $\eta$ is alternating we call $\cP_k$ a \emph{symplectic} $k$-\emph{Grassmannian} and denote it $\cS_k(n,d;\FF)$ (recall that if $\eta$ is alternating then $d_0 = 0$, so we may omit to keep a record of $d_0$). Finally, when $\cP_k$ is associated to a non-degenerate quadratic form, then it is called \emph{orthogonal} $k$-\emph{Grassmannian}. We denote it by $\cQ_k(n,d_0,d;\FF)$.

We will also often use the following shortenings: $\cH_k$ for $\cH_k(n, d_0,d;\FF)$, $\cS_k$ for $\cS_k(n,d;\FF)$ and $\cQ_k$ for $\cQ_k(n,d_0,d;\FF)$. 
\end{definition}

In this paper we shall compute the dimensions of the Grassmann embeddings of $\cH_k(n,d_0,d;\FF)$, $\cS_k(n,d;\FF)$ and $\cQ_k(n,d_0,d;\FF)$ for any $k$, with no hypotheses on either $d$ or $d_0$. Note that, in general, the possible values that the anisotropic defect $d_0$ can take depend on the field $\FF$. For instance if $\FF$ is finite and $\eta$ is Hermitian then $d_0\leq 1$. If $\FF$ is quadratically closed, then a quadratic form defined over $\FF$ has anisotropic defect $d_0 \leq 1$ and if $\FF$ is finite then $d_0\leq 2$. 

\begin{remark}
In the literature the word ``defect'' is sometimes given a meaning different from either of those stated above. Indeed a number of authors use it to denote the defect of the bilinearization of a non-degenerate quadratic form in characteristic $2$. We shall consider this defect in Subsections \ref{Main Results Sec}, \ref{even char} and \ref{Wg}, denoting it by the symbol $d'_0$. Clearly, $d'_0 \leq d_0$.
\end{remark}

\begin{remark}
A number of authors (Tits \cite[Chapter 8]{Tits}, for instance), when dealing with vectors (or subspaces) that are singular (totally singular) for a given sesquilinear form, prefer the word ``isotropic'' rather than ``singular'',  keeping the latter only for pseudoquadratic forms. Other authors (e.g. Buekenhout and Cohen \cite[Chapters 7--10]{BuekC}) use ``singular'' in any case. We have preferred to follow these latter ones.
\end{remark}

\subsection{A survey of known results}\label{Survey}

Before stating our main results,  we provide a brief summary of what is currently known about the dimension of the Grassmann embedding of a
polar $k$-Grassmannian. In this respect, only non-degenerate forms are considered in the literature. So, throughout this subsection we assume $d = 0$.

We consider $\cH_k(n,d_0,0;\FF)$ first. The dimension of the Grassmann embedding of $\cH_k(n,d_0,0;\FF)$ has been proved to be equal to ${N\choose k}$ for $d_0 = 0$ and $k$ arbitrary by de Bruyn~\cite{B10} (see also Blok and Cooperstein~\cite{BC2012}) and for $d_0$ arbitrary and $k = 2$ by Cardinali and Pasini~\cite{IP14}. When $k = 1$ there is nothing to say: $\varepsilon_1$ is just the canonical embedding of the polar space $\cH_1(n,d_0,0;\FF)$ in
$\PG(V)$. As far as we know, the case where $k > 2$ and $d_0 > 0$ has never been considered so far.

It is worth to spend a few words on $\cH_n = \cH_n(n,d_0,0;\FF)$. When $d_0=0$ the Grassmann embedding $\varepsilon_n$ of $\cH_n$ is laxly projective:
it maps the lines of $\cH_n$ onto Baer sublines of $\PG(\bigwedge^nV)$; by replacing $\PG(\bigwedge^nV)$ with a suitable Baer subgeometry containing $\varepsilon_n(\cH_n)$, the embedding $\varepsilon_n$ is turned into a projective embedding (see e.g.~\cite{B10} or  \cite{CS01};
also~\cite[Section 4]{IBP07}). This modification has no effect on the dimension of $\varepsilon_n$, which remains the same. On the other hand, when $d_0>0$ then $\varepsilon_n$ maps the lines of $\cH_n$ onto Hermitian hypersurfaces in $(d_0+1)$-dimensional subspaces of $\PG(\bigwedge^nV)$ (Hermitian plane curves when $d_0=1$). In this case $\cH_n$ does not admit any projective embedding, as it follows from the classification of Moufang quadrangles (Tits and Weiss \cite{TW}).

As for $\cS_k(n,0;\FF)$, it is well known that its Grassmann embedding has dimension ${N\choose k}-{N\choose{k-2}}$, with the usual convention that ${N\choose{-1}} := 0$ in the case $k = 1$; see e.g. see De Bruyn~\cite{B09} or Premet and Suprunenko~\cite{Premet} (also De Bruyn~\cite{B07}, Cooperstein~\cite{C98}).

Let now $\varepsilon_k$ be the Pl\"ucker embedding of $\cQ_k(n,d_0,0;\FF)$. The dimension of $\varepsilon_k$ is known only for $d_0\leq 1$ with the further restriction $k<n$ when $d_0=0$. Indeed, in~\cite{IP13} it is shown that
  \[ \dim(\varepsilon_k)=\begin{cases}
      {{N}\choose{k}} & \text{ if $\ch(\FF)$ is odd and $d_0\leq 1$.} \\
      \binom{N}{k}-\binom{N}{k-2} & \text{ if $\ch(\FF)$ is even and $d_0\leq 1$.}
    \end{cases}\]
The embedding $\varepsilon_n$ of $\cQ_n = \cQ_n(n,d_0,0;\FF)$ also deserves a few comments. When $d_0 = 0$ the lines of $\cQ_n$ are just pairs of points. This case does not look very interesting. Suppose $d_0>0$. Then $\varepsilon_n$ maps the lines of $\cQ_n$ onto non-singular quadrics in $(d_0+1)$-dimensional subspaces of $\PG(\bigwedge^nV)$ (conics for $d_0=1$ and elliptic ovoids for $d_0=2$). If $d_0=1$ then $\cQ_n$ admits the so-called \emph{spin embedding}, which is projective and $2^n$-dimensional. Interesting relations exist between this embedding and $\varepsilon_n$ (see~\cite{IP13}, \cite{IP13bis}; also Section~\ref{Wg k=n} of this paper). Furthermore, still assuming $d_0 = 1$, if $\FF$ is a perfect field of characteristic $2$ then $\cQ_n \cong \cS_n = \cS_n(n,0;\FF)$. In this case the Grassmann embedding of $\cS_n$ yields a projective embedding of $\cQ_n$ which, as proved in~\cite{IP13}, is a quotient of $\varepsilon_n$. A $2^n$-dimensional projective embedding also exists for $\cQ_n$ when $d_0=2$ (see e.g. Cooperstein and Shult~\cite[\S 2.2]{CS01}). It is likely that some relation also exists between this embedding and $\varepsilon_n$ (see Section~\ref{Wg k=n}). If $d_0>2$, then $\cQ_n$ admits no projective embedding, as it follows from~\cite{TW}. 

\subsection{The main results of this paper}\label{Main Results Sec}

In Sections \ref{Hermitian}, \ref{Bsymp} and \ref{orthogonal} of this paper we shall compute the dimension of the Grassmann embedding of a polar $k$-Grassmannian for $k$ not greater than the reduced Witt index $n$ of the associated form but with no restrictions on the anisotropic and singular defects $d_0$ and $d$. As a by-product, we obtain anew the results mentioned in the previous subsection. Explicitly, we prove the following:

\begin{MainTheo}\label{Main Theorem}
Let $V$ be a vector space of finite dimension $N > 1$ over a field $\FF$ and let $n, d_0$ and $d$ be non-negative integers with $n > 1$, $0\leq d, d_0$ and $2n+d_0+d = N$. Let $\eta$ be a Hermitian, alternating or quadratic form defined over $V$, with reduced Witt index $n$, anisotropic defect $d_0$ and singular defect $d$, provided that such a form exists. For $1\leq k \leq n$ let $\cP_k$ be the polar $k$-Grassmannian associated to $\eta$ and $\varepsilon_k$ its Grassmann embedding. Then:
\begin{enumerate}
\item\label{pt1} If $\eta$ is Hermitian then $\dim(\varepsilon_k) = {N\choose{k}}$, namely $\varepsilon_k(\cP_k)$ spans $\mathrm{PG}(\bigwedge^kV)$.
\item\label{pt2} If $\eta$ is alternating then $\dim(\varepsilon_k) = {N\choose k}-{N\choose{k-2}}$.
\item\label{pt3} Suppose $\eta$ to be a quadratic form. Then $\dim(\varepsilon_k) = {N\choose{k}}$ if $\ch(\FF) \neq 2$ and $\dim(\varepsilon_k) = {N\choose{k}} - {N\choose{k-2}}$ if $\ch(\FF) = 2$. In other words, if either $\ch(\FF) \neq 2$ or $k = 1$ then $\varepsilon_k(\cP_k)$ spans $\mathrm{PG}(\bigwedge^kV)$, otherwise $\langle\varepsilon_k(\cP_k)\rangle$ coincides with the span of the image of the Grassmann embedding of the symplectic $k$-Grassmannian associated with the bilinearization of $\eta$.
 \end{enumerate}
\end{MainTheo}

\begin{remark}
As noticed in Subsection \ref{Survey}, the dual polar space $\cQ_n(n,0,0;\FF)$ is not considered in \cite{IP13}. Part \ref{pt3} of the above theorem includes this case too.
\end{remark}

In Theorem \ref{Main Theorem} we have assumed $k \leq n$, but $n < k \leq n+d$ is also allowed by the definition of polar Grassmannian when $d > 0$. We consider this case in the next corollary, to be proved in Section~\ref{nova Sec}. 

\begin{MainCor}\label{Main Corollary} 
With the notation of Theorem {\rm \ref{Main Theorem}}, let $n < k \leq n+d$.  
\begin{enumerate}
\item\label{cc1} If $\eta$ is Hermitian or quadratic, but with $\ch(\FF) \neq 2$ in the latter case, then 
\[\dim(\langle\varepsilon_k(\cP_k)\rangle) = {N\choose k} -\sum_{i=0}^{k-n-1}{{N-d}\choose{k-i}}{d\choose i},\]
with the usual convention that a binomial coefficient ${m\choose h}$ is $0$ when $h > m$. 
\item\label{cc2}  If $\eta$ is alternating or quadratic, with $\ch(\FF) = 2$ in the latter case, then
\[\dim(\langle\varepsilon_k(\cP_k)\rangle) = {N\choose k} - {N\choose{k-2}} - \sum_{i=0}^{k-n-1}{{N-d}\choose{k-i}}{d\choose i} + \sum_{i=0}^{k-n-1}{{N-d}\choose{k-i-2}}{d\choose i}.\]
\end{enumerate} 
In any case $\langle\varepsilon_k(\cP_k)\rangle$ is a proper subspace of $\mathrm{PG}(\bigwedge^kV)$.
\end{MainCor}

\begin{remark}
  A linear system of ${N\choose{k-2}}$ equations is proposed in~\cite[Section 4]{Pas16} which, combined with the (non-linear) equations describing the Grassmann variety $e_k(\cG_k)$, characterizes the variety $\bS_k = \varepsilon_k(\cS_k)$. It is asked in~\cite{Pas16} whether those linear equations are linearly independent and if they characterize $\langle\bS_k\rangle$. In general, the answer to either of these questions is negative. In view of Corollary~\ref{Main Corollary}, the answer is certainly negative when $k > n$. However, it is negative even if $k \leq n$ (in particular, when $d = 0$). For instance, let $d = 0$. Then if $k \leq 3$ those equations are independent, whence they indeed describe $\langle\bS_k\rangle$; perhaps  they are independent for any $k \leq n$
when $\ch(\FF) \neq 2$,
but when $\ch(F) = 2$ and $k > 3$ they are dependent, as one can see by a straightforward check. So, here is one more problem: find a linear system that describes $\langle\bS_k\rangle$.
\end{remark}

We now turn to the second problem studied in this paper. Let $\cQ_k = \cQ_k(n,d_0,0;\FF)$ be a non-degenerate orthogonal $k$-Grassmannian with $k < n$. As $k < n$, the Grassmann embedding $\varepsilon_k$ of $\cQ_k$ is projective. Assume moreover that $\ch(\FF) = 2$. So $\dim(\varepsilon_k)  = {N\choose{k}} - {N\choose{k-2}}$ by Theorem \ref{Main Theorem}. As proved in \cite{IP13}, if $k > 1$ and $d_0 \leq 1$ then $\varepsilon_k$ is a proper quotient of an ${N\choose k}$-dimensional projective embedding $e^W_k$ of $\cQ_k$, which in~\cite{IP13} is called \emph{Weyl embedding} and lives in the Weyl module $W(\mu_k)$ for the Chevalley group $G = \mathrm{O}(N,\FF)$ (when $d_0 = 1$) or $G = \mathrm{O}^+(N,\FF)$ (when $d_0 = 0$), for a suitable weight $\mu_k$ of the root system of $G$ (see e.g.~\cite{H72} or~\cite{S67} for these notions). Explicitly, $\mu_k$ is the $k$-th fundamental weight $\lambda_k$ except when $k = n-1$ and $d_0 = 0$. In the latter case $\mu_{n-1} = \lambda_{n-1}+\lambda_n$. We refer to~\cite{IP13} (also~\cite{IP14b}) for more information on $e^W_k$. We only recall here that the Weyl embedding $e^W_k$ also exists when $\ch(\FF) \neq 2$ and when $k = 1$, however $e^W_k \cong \varepsilon_k$ for any $k < n$ when $\ch(\FF) \neq 2$ (as proved
in~\cite{IP13}) and $e^W_1 \cong \varepsilon_1$ whatever $\ch(\FF)$ is. Note also that, in any case, $\dim(W(\mu_k)) = {N\choose k}$. So, $e^W_k(\cQ_k)$ spans $\PG(W(\mu_k))$.

It is natural to ask whether an analogue of the Weyl embedding can be defined when $d_0 > 1$. In the last section of this paper we propose such a generalization, but only for orthogonal Grassmannians associated to absolutely non-degenerate quadratic forms.

We recall that a non-degenerate quadratic form $q:V\rightarrow\FF$ is said to be \emph{absolutely non-degenerate} if its natural extension $\overline{q}:\overline{V}\rightarrow\overline{\FF}$ to $\overline{V} := V\otimes\overline{\FF}$, where $\overline{\FF}$ is the algebraic closure of $\FF$, is still non-degenerate. When $\ch(\FF) \neq 2$ this notion is devoid of interest: in this case all non-degenerate quadratic forms are absolutely non-degenerate. On the other hand, when $\ch(\FF) = 2$, let $f_q$ be the bilinearization of $q$ and $d'_0 := \dim(\mathrm{Rad}(f_q))$. Then $\dim(\mathrm{Rad}(\overline{q})) = \max(0, d'_0-1)$. So, $q$ is absolutely non-degenerate if and only if $d'_0 \leq 1$.

Let $q$ be absolutely non-degenerate with anisotropic defect $d_0$. Then, as we shall prove in Section~\ref{Wg}, the field $\FF$ admits an algebraic extension $\widehat{\FF}$
such that the extension $\widehat{q}:\widehat{V}\rightarrow\widehat{\FF}$ of $q$ to $\widehat{V} = V\otimes\widehat{\FF}$ is non-degenerate with anisotropic defect $\mathrm{def}_0(\widehat{q}) = d''_0 \leq 1$, where $d''_0 = d'_0$ when $\ch(\FF) = 2$ while, if $\ch(\FF) \neq 2$, then $d''_0 = 0$ or $1$ according to whether $N$ is even or odd. In any case, keeping the hypothesis $k < n$, the $k$-Grassmannian $\widehat{\cQ}_k$ associated to $\widehat{q}$ admits the Weyl embedding $\widehat{e}^W_k:\widehat\cQ_k\rightarrow\mathrm{PG}(\widehat{V}^W_k)$, where $\widehat{V}^W_k$ is the appropriate Weyl module. Clearly, the orthogonal $k$-Grassmannian $\cQ_k$ associated to $q$ is a subgeometry of $\widehat\cQ_k$, $\PG(\bigwedge^kV)$ is a subgeometry of $\PG(\bigwedge^k\widehat{V})$ and the Grassmann embedding $\widehat{\varepsilon}_k:\widehat\cQ_k\rightarrow\mathrm{PG}(\bigwedge^k\widehat{V})$ of $\widehat\cQ_k$ induces on $\cQ_k$ its Grassmann embedding $\varepsilon_k:\cQ_k\rightarrow\mathrm{PG}(\bigwedge^kV)$.

The following theorem will be proved in Section~\ref{Wg}. In order to make its statement a bit shorter, we take the liberty of using the symbols $\langle\widehat{\varepsilon}_k(\widehat\cQ_k)\rangle$ and $\langle\varepsilon_k(\cQ_k)\rangle$, which actually stand for subspaces of $\PG(\bigwedge^k\widehat{V})$ and $\PG(\bigwedge^kV)$ respectively, to denote the corresponding vector subspaces of $\bigwedge^k\widehat{V}$ and $\bigwedge^kV$.

\begin{MainTheo}\label{MainTheorem bis}
Let $q$ be absolutely non-degenerate, $k < n$ and let $\widehat\cQ_k$, $\widehat{V}$, $\widehat{V}^W_k$, $\widehat{\varepsilon}_k$ and $\widehat{e}^W_k$ be as defined above. Then the Weyl module $\widehat{V}^W_k$, regarded as an $\FF$-space, contains an $\FF$-subspace $V^W_k$ such that:
\begin{enumerate}
\item The Weyl embedding $\widehat{e}^W_k$ induces on $\cQ_k$ a projective embedding $e^W_k:\cQ_k\rightarrow\PG(V^W_k)$. Moreover $e^W_k(\cQ_k)$ spans $V^W_k$.
\item The (essentially unique) morphism $\widehat{\varphi}:\widehat{V}^W_k\rightarrow\langle\widehat{\varepsilon}_k(\widehat\cQ_k)\rangle$ from the Weyl embedding $\widehat{e}^W_k$ to the Grassmann embedding $\widehat{\varepsilon}_k$ of $\widehat\cQ_k$ maps $\widehat{V}^W_k$ onto $\langle\widehat{\varepsilon}_k(\widehat\cQ_k)\rangle$ and induces a morphism from $e^W_k$ to the Grassmann embedding $\varepsilon_k$ of $\cQ_k$.
\item If $\widehat{\varphi}$ is an isomorphism (which is the case precisely when either $\ch(\FF) \neq 2$ or $k = 1$) then $\varphi$ also is an isomorphism.
\end{enumerate}
\end{MainTheo}

We call $e^W_k$ the \emph{Weyl-like} embedding of $\cQ_k$ (Subsection~\ref{Wg}, Definition~\ref{Weyl-like}). Needless to say, when $d_0 \leq 1$ then $e^W_k$ is just the Weyl embedding of $\cQ_k$.
Clearly, when $\varphi$ is an isomorphism, the Weyl-like embedding $e^W_k \cong\varepsilon_k$ is ${N\choose k}$-dimensional. Otherwise, as we shall prove in Subsection~\ref{Wg},
\[{N\choose k}-{N\choose{k-2}} ~ \leq ~ \dim(e_k^W) ~ \leq ~ {N\choose k}+ {N\choose{k-2}}(g-1)\]
where $g := |\widehat{\FF}:\FF|$. As we shall see in Subsection~\ref{Wg}, we can always choose $\widehat{\FF}$ in such a way that $g \leq \max(1, d_0-d'_0)$ (note that we are assuming $\ch(\FF) = 2$, otherwise $\varphi$ is an isomorphism). However, even with $g \leq \max(1, d_0-d'_0)$, the above bounds are likely to be rather lax. We leave the task of improving them for a future work.

\medskip

To finish, we mention an important problem which stands in the background of this paper: under which conditions the embeddings considered in this paper are universal? Of course, this question makes sense only if universality can be defined in a sensible way for the family of embeddings we consider, as when they are projective (but not only in that case). Apart from the trivial case of $k = 1$, where $\varepsilon_1$ is just the canonical embedding of the polar space ${\cP} = {\cP}_1$, which is indeed universal except when ${\cP}$ is symplectic and $\ch(\FF) = 2$, sticking to non-degenerate cases, a clear answer was known only for $\cH_k(n,0,0;\FF)$ and $\cS_k(n,0;\FF)$: the Grassmann embedding of $\cH(n,0,0;\FF)$ is universal provided that $|\FF| > 4$ and that of $\cS_k(n,0;\FF)$ is universal provided that $\ch(\FF) \neq 2$ (Blok and Cooperstein~\cite{BC2012}). Partial answers for $\cQ_k(n,d_0,0;\FF)$ with $d_0 \leq 1$ are also known, which might suggest that $e^W_k$ is universal when $k < n$ and $\varepsilon_n$ is universal when $\ch(\FF)\neq 2$  (see e.g. \cite[Theorem 1.5]{IP13} for $1 < k \leq 3$, $k < n$ and \cite[Theorem 5]{IP13bis} for $k = n = 2$).
In a recent paper~\cite{ILP19}, the authors have investigated
the generating rank of polar Grassmannians; in particular, for
$\cH_k(n,d_0,0;\FF)$ with $d_0 \geq 0$ and $k<n$ it is shown that
the Grassmann embedding of $\cH_k$ is universal; see~\cite[Corollary 2]{ILP19}.
For $k=2$, and $k=3<n$ for $d_0\leq1$, the Grassmann embedding of $\cQ_k(n,d_0,0,\FF)$ is universal; see~\cite{IP13}.
Very little is known
for $\cQ_k(n,d_0,0;\FF)$ with $d_0 > 1$ or $k>2$.
However, we are not going to further address this problem in the present paper.

\paragraph{Structure of the paper}
In Section~\ref{preliminaries} we set some notation and prove some preliminary general results. Parts \ref{pt1}, \ref{pt2} and \ref{pt3} of Theorem \ref{Main Theorem} will be proved in Sections \ref{Hermitian},
\ref{Bsymp} and \ref{orthogonal} respectively. Finally, in Section~\ref{Weyl generalizzato} we propose a general definition of \emph{liftings} of embeddings and use this notion to prove Theorem~\ref{MainTheorem bis}.

\section{Preliminaries}\label{preliminaries}

Let $V:=V(N,\FF)$ be a vector space of dimension $N$ over a field $\FF$. Let $E:=(e_i)_{i=1}^N$ be a given basis of $V$.
For any set $J=\{j_1,\dots,j_h\}$ of indexes with $1\leq j_i\leq N$ denote by $V_J$ the subspace of $V$ generated by $E_J:=(e_{j_1},\dots,e_{j_h})$.

We shall write in brief $V_k:=\bigwedge^kV$ and $V_{J,k}:=\bigwedge^kV_J$. It is well known that a basis $E_k$ of $V_k$ is given by all vectors of
the form $e_T:=e_{t_1}\wedge e_{t_2}\wedge\dots\wedge e_{t_k}$ with $t_1<t_2<\dots<t_k$ where $T=\{t_1,\dots,t_k\}$ varies among all $k$-subsets of $\{1,\dots,N\}$. Consistently with the notation introduced above we shall write $E_{J,k}$ for the basis of $V_{J,k}$ induced by $E_J$.

Given a Hermitian, alternating, or quadratic form $\eta$ defined over $V$, for any $J\subseteq\{1,2,\dots,N\}$ let $\eta_J$ be the restriction of $\eta$ to $V_J$. A $k$-Grassmannian associated to $\eta_J$ will be denoted by $\cP_{J,k}$. We shall write the image of  $\cP_k$ under its Grassmann embedding $\varepsilon_k$ as:
\[\PP_k:=\{\varepsilon_k(X)\colon X \mbox{ is a point  of }\cP_k\}\subseteq \PG(V_k).\]
According to the notation introduced above we  put
\[\PP_{J,k}:=\{\varepsilon_k(X)\colon X \mbox{ is a point of }\cP_{J,k}\}\subseteq \PG(V_{J,k}).\]
In particular,
\[\begin{array}{l@{}c@{}l@{}l}
\HH_{J,k}&:=&\{\varepsilon_k(X)\colon X \mbox{\rm\ is a point of }\cH_{J,k}&\},\\
\mathbb{S}_{J,k}&:=&\{\varepsilon_k(X)\colon X \mbox{\rm\ is a point of }\cS_{J,k}&\},\\
    \QQ_{J,k}&:=&\{\varepsilon_k(X)\colon X \mbox{\rm\ is a point of }\cQ_{J,k}&\},\\
  \end{array}\]
for respectively the image of  Hermitian, alternating and orthogonal Grassmannians. If $J=\{1,\dots, N\}$ then $\HH_{J,k}$, $\mathbb{S}_{J,k}$  and $\QQ_{J,k}$ will just  be denoted by $\HH_{k}$, $\mathbb{S}_k$ and $\QQ_{k}$ respectively. We have defined  $\PP_k$ and $\PP_{J,k}$ as sets of points in $\PG(V_k)$ and $\PG(V_{J,k})$. In the sequel,
with some abuse of notation, we shall often take the liberty to regard them also as sets of vectors of respectively $V_k$ and $V_{J,k}$,
 implicitly replacing $\{\varepsilon_k(X) : X \mbox{ is a point of } \cP_k \}$ with $\{v\in \varepsilon_k(X) : X \mbox{ is a point of } \cP_k \}$ and
similarly for $\PP_{J,k}$. With these conventions, we can define
\[ \PP_{{I},h}\wedge\PP_{{J},k}:=\{\langle v\wedge w\rangle \colon \langle v\rangle\in \PP_{{I},h} \mbox{\rm\  and } \langle w\rangle\in \PP_{{J},k} \}\subseteq\PG(V_{I\cup J,h+k}). \]
We always regard $\langle\PP_k\rangle$ and $\langle\PP_{J,k}\rangle$ as subspaces of $V_k$ and $V_{J,k}$ respectively (as we did in the Introduction when defining $\dim(\varepsilon_k)$).

\subsection{Orthogonal decompositions}

As above, for $I, J \subseteq \{1, 2,\dots, n\}$ let $\eta_I$ and $\eta_J$ be the forms induced by $\eta$ on $V_I$ and $V_J$ respectively.  We put $d_I := \mathrm{def}(\eta_I)$, $d_J := \mathrm{def}(\eta_J)$ and we denote by $n_I$ and $n_J$ the reduced Witt index of $\eta_I$ and $\eta_J$.

\begin{lemma}\label{l11}
Suppose that $I\cap J = \emptyset$ and $V_I\perp V_J$ with respect to $\eta$ (or its bilinearization if $\eta$ is quadratic). Then, for $1 \leq h \leq n_I+d_I$ and $1\leq k\leq n_J+d_J$ we have $\PP_{{I},h}\wedge\PP_{{J},k}\subseteq\PP_{{I}\cup {J},h+k}$ and $\langle\PP_{I,h}\rangle\wedge\langle\PP_{J,k}\rangle\subseteq \langle\PP_{I\cup J,h+k}\rangle$.
\end{lemma}
\begin{proof} Take $\langle v\rangle\in\PP_{I,h}$ and $\langle w\rangle\in\PP_{J,k}$.
Since $V_I$  and $V_J$ are orthogonal by hypothesis, the $h$-dimensional vector space $X_v:=\varepsilon_h^{-1}(v)$ and the $k$-dimensional vector space $X_w:=\varepsilon_k^{-1}(w)$  are mutually orthogonal. Hence the space $X_v+X_w$ is totally singular and it has dimension $h+k$ (recall that $V_I\cap V_J = \{0\}$ as $I\cap J = \emptyset$ by assumption). So,
$\langle v\wedge w\rangle=\varepsilon_{h+k}(\langle X_v,X_w\rangle)\in\PP_{I\cup J,h+k}$. The condition on the linear spans is now immediate.
\end{proof}
In the hypotheses of Lemma~\ref{l11} the form $\eta_{I\cup J}$ induced by $\eta$ on $V_{I\cup J}$ is the orthogonal sum of $\eta_I$ and $\eta_J$. Accordingly, $d_I+d_J = \mathrm{def}(\eta_{I\cup J})$, namely $\mathrm{Rad}(\eta_{I\cup J}) = \mathrm{Rad}(\eta_I)\oplus\mathrm{Rad}(\eta_J)$. Moreover $n_I+n_J$ is the reduced Witt index of $\eta_{I\cup J}$. Clearly, $n_I+n_J \leq n$ but no relation can be stated between $d_I+d_J$ and $d$ in general. Indeed, although we always have $V_{I\cup J}\subseteq\mathrm{Rad}(\eta_{I\cup J})^\perp$, in general  $\mathrm{Rad}(\eta_{I\cup J})^\perp \subset V$.

\subsection{Reduction to the non-degenerate case}

As above, let $\eta$ be a Hermitian, alternating, or quadratic form defined over $V$, with $\dim(V) = N$. Let $R := \mathrm{Rad}(\eta)$, $d = \dim(R) = \mathrm{def}(\eta)$ and let $n$ be the reduced Witt index of $\eta$.

With $V_0, V_1,\dots, V_n$ as in decomposition \eqref{decomposizione di base}, let $\overline{V} := V_0\oplus V_1\oplus \dots \oplus V_n$ and let $\overline{\eta}$ be the form induced by $\eta$ on $\overline{V}$. Note that $\dim(\overline{V}) = 2n+d_0 = N-d$, the form $\overline{\eta}$ is non-degenerate and it is isomorphic to the reduction of $\eta$, namely the form induced by $\eta$ on $V/R$.

Given $k\leq n$, for $1\leq j\leq k$ the polar $j$-Grassmannian $\overline{\cP}_j$ associated to $\overline{\eta}$ is a full subgeometry of $\cP_j$. Its $\varepsilon_j$-image $\overline{\PP}_j = \varepsilon_j(\overline{\cP}_j)$ is contained in $\overline{V}_j := \bigwedge^j\overline{V}$, which is a subspace of $V_j$.

\begin{lemma}\label{prel novo 1}
For $1\leq k\leq n$ we have
\[\langle\PP_k\rangle = \bigoplus_{i=0}^{\min(d,k)}\langle\overline{\PP}_{k-i}\rangle\wedge\bigwedge^iR\]
where $\bigwedge^0R := \FF$ (as usual) and $\langle\overline{\PP}_0\rangle := \bigwedge^0\overline{V} = \FF$ by convention (when $d \geq k$).
\end{lemma}
\begin{proof} Every vector $x\in V$ splits as $\overline{x}+x^R$ for uniquely determined vectors $\overline{x}\in\overline{V}$ and $x^R\in R$. Moreover, $x$ is $\eta$-singular if and only if $\overline{x}$ is $\overline{\eta}$-singular and $x \perp y$ where $y = \overline{y}+y^R$ if and only if $\overline{x}\perp \overline{y}$. It follows that every vector of $\PP_k$ is a sum $\sum_{i=0}^{\min(d,k)}u_i\wedge v_i$ with $u_i\in \overline{\PP}_{k-i}$ and $v_i$ a pure power in $\bigwedge^iR$, for $i = 0, 1,\dots, \min(d,k)$. Conversely, every wedge product $u_i\wedge v_i$ as above belongs to $\PP_k$. The conclusion follows from these remarks and the fact that, since $V = \overline{V}\oplus R$, we also have
$\bigwedge^kV = \bigoplus_{i=0}^{\min(d,k)}\bigwedge^{k-i}\overline{V}\wedge\bigwedge^iR$.
\end{proof}

\begin{corollary}\label{prel novo 2}
For $1\leq k \leq n$ we have
\[\dim(\langle\PP_k\rangle) = \sum_{i=0}^{\min(d,k)}\dim(\langle\overline{\PP}_{k-i}\rangle)\cdot{d\choose i}.\]
\end{corollary}
\begin{proof}
It follows directly from Lemma~\ref{prel novo 1}, recalling that $\dim(\bigwedge^iR) = {d\choose i}$ and $\dim(\langle X\wedge Y\rangle) = \dim(X)\cdot\dim(Y)$ for any two vector spaces $X$ and $Y$ with trivial intersection.
\end{proof}

We recall the following property of binomial coefficients.
\begin{lemma}\label{lemma-sum}
We have:
  \begin{equation}\label{e-sum}
 \sum_{i=0}^{\infty}{{N-d}\choose{k-i}}{d\choose i}={N\choose k}
  \end{equation}
 where, as usual, we put ${{N-d}\choose{k-i}} :=0$ if either $k-i < 0$ or $k-i > N-d$. 
\end{lemma}
% \begin{proof} Put $M := N-d$. Thus~\eqref{e-sum} can be rewritten as 
%  \begin{equation}\label{e-sum bis}
%  \sum_{i=0}^{\infty}{M\choose{k-i}}{d\choose i}={{M+d}\choose k}.
%   \end{equation}
% We shall prove by induction on $d$ that \eqref{e-sum bis} holds true for any choice of $k$ and $M$. The case $d=0$ is immediate. Suppose \eqref{e-sum bis} holds for a given value of $d$ with $k$ and $M$ arbitrary and consider  $d+1$; since
% ${{d+1}\choose i}={d\choose i}+{d\choose{i-1}}$ (with ${d\choose -1} = 0$ when $i = 0$), by the inductive hypothesis we get
%   \[  \sum_{i=0}^{\infty}{M\choose{k-i}}{{d+1}\choose i}= \sum_{i=0}^{\infty}{{M}\choose{k-i}}\left({d\choose i}+{d\choose{i-1}}
%     \right)={{M+d}\choose{k}}+\sum_{i=0}^{\infty}{{M}\choose{k-i}}{d\choose{i-1}}.\]
% Also
%   \[ \sum_{i=0}^{\infty}{{M}\choose{k-i}}{d\choose{i-1}}= \sum_{j=0}^{\infty}{{M}\choose{k-1-j}}{d\choose{j}}= {{M+d}\choose{k-1}}, \]
% again by the inductive hypothesis but with $k$ replaced by $k-1$. So,
%   \[ \sum_{i=0}^{\infty}{{M}\choose{k-i}}{{d+1}\choose i}=  {{M+d}\choose k}+{{M+d}\choose{k-1}}={{M+d+1}\choose k}.\]
%   The inductive step is done. Hence~\eqref{e-sum bis} holds.
%   Consequently~\eqref{e-sum} holds as well.
% \end{proof}
\begin{proof}
  By the binomial theorem,
  the coefficient of $x^k$ in $(x+1)^N$ is the right hand side
  of~\eqref{e-sum}; however
  $(x+1)^N=(x+1)^{N-d}(x+1)^d$ and the coefficient of $x^k$ in
  $(x+1)^{N-d}(x+1)^d$ is the left hand side of~\eqref{e-sum}.
\end{proof}

\begin{theorem}\label{prel novo 3}
Both of the following hold for $1 \leq k \leq n$.
\begin{enumerate}
\item\label{pt1a} If $\dim(\langle\overline{\PP}_h\rangle) = {{N-d}\choose h}$ for every $h = 1, 2,\dots, n$ then $\dim(\langle\PP_k\rangle) = {N\choose k}$.
\item\label{pt2a} If $\dim(\langle\overline{\PP}_h\rangle) = {{N-d}\choose h}-{{N-d}\choose{h-2}}$ for every $h = 1, 2,\dots, n$, then $\dim(\langle\PP_k\rangle) = {N\choose k}-{N\choose{k-2}}$.
\end{enumerate}
\end{theorem}
\begin{proof}
To prove Part \ref{pt1a}, replace $\dim(\langle\overline{\PP}_{k-i}\rangle)$ with ${{N-d}\choose{k-i}}$ in Corollary \ref{prel novo 2} and apply \eqref{e-sum} of Lemma \ref{lemma-sum}. Turning to Part \ref{pt2a}, replace $\dim(\langle\overline{\PP}_{k-i}\rangle)$ with ${{N-d}\choose{k-i}}-{{N-d}\choose{k-i-2}}$ in Corollary \ref{prel novo 2}. We obtain
\begin{multline*}
  \dim(\langle\PP_k\rangle) = \sum_{i=0}^{\min(d,k)}\left({{N-d}\choose{k-i}}-{{N-d}\choose{k-i-2}}\right)\cdot{d\choose i} = \\
  \sum_{i=0}^{\min(d,k)}{{N-d}\choose{k-i}}{d\choose i} - \sum_{i=}^d{{N-d}\choose{k-2-i}}{d\choose i}.
\end{multline*}
Hence $\dim(\langle\PP_k\rangle) = {N\choose k}-{N\choose{k-2}}$ by Lemma \ref{lemma-sum}.
\end{proof}

\section{Hermitian $k$-Grassmannians }\label{Hermitian}
In this section we shall prove Part~\ref{pt1} of Theorem~\ref{Main Theorem}. In view of Part~\ref{pt1a} of Theorem~\ref{prel novo 3}, it is sufficient to prove Part~\ref{pt1} of Theorem~\ref{Main Theorem} in the non-degenerate case. Accordingly, throughout this section $h:V\times V\to\FF$ is a non-degenerate $\sigma$-Hermitian form of Witt index $n$ and anisotropic defect $\mathrm{def}_0(h) = d_0 = N-2n$, where $N = \dim(V)$ and $\sigma$ is an involutory automorphism of $\FF$. Let $\FF_0:=\Fix(\sigma)$ be the subfield of $\FF$ consisting of the elements fixed by $\sigma$.
It is always possible to choose a basis
\[ E=(e_1,\dots,e_{2n},e_{2n+1},\dots,e_N) \]
of $V$ and $\kappa_{2n+1},\dots,\kappa_{N} \in \FF_0\setminus\{0\}$ such that
\begin{equation}\label{hermform}
h\left(\sum_{i=1}^Ne_ix_i,\sum_{j=1}^Ne_jy_j\right)= \sum_{i=1}^n(x_{2i-1}^{\sigma}y_{2i}+x_{2i}^{\sigma}y_{2i-1})+\sum_{j=2n+1}^N
\kappa_jx_j^{\sigma}y_j
\end{equation}
where the form induced by $h$ on $\langle e_{2n+1},\dots,e_N\rangle$ is anisotropic in $\FF$, that is
\[ \sum_{j=2n+1}^N\kappa_jx_j^{\sigma+1}=0 \Leftrightarrow x_{2n+1}=x_{2n+2}=\dots =x_N=0, \]
see~\cite[\S 6]{BA9}.
Observe that for all $i\in\{1,\dots,n\}$ the vectors $(e_{2i-1},e_{2i})$ form a hyperbolic pair for $h$.

For $1\leq k\leq n$, let $\cH_k$ be the Hermitian $k$-Grassmannian associated to $h$ and $\HH_k = \varepsilon_k(\cH_k)$ be its $\varepsilon_k$-image in $V_k$. According to the conventions stated in Section \ref{preliminaries}, given $J\subseteq\{1,2,\dots, N\}$ we denote by $h_J$ the restriction of $h$ to $V_J\times V_J$, by $\cH_{J,k}$ the Hermitian $k$-Grassmannian associated to $h_J$ (if $k$ is not greater than the Witt index of $h_J$) and we put $\HH_{J,k} = \varepsilon_k(\cH_{J,k})$ ($\subseteq V_{J,k}$).

\begin{lemma}\label{l2}
Suppose $n\geq 2$. Given $i,j\in\{1,2,\dots,n\}$ with $i\neq j$, let $J=\{2i-1,2i,2j-1,2j\}$. Then,  $\langle\HH_{J,2}\rangle=V_{J,2}$.
\end{lemma}
\begin{proof}
Clearly, $\langle\HH_{J,2}\rangle\subseteq V_{J,2}$. Recall that $V_{J,2}=\bigwedge^2\langle e_{2i-1},e_{2i},e_{2j-1},e_{2j}\rangle$. By definition of $h$, the vectors $e_{2i-1}\wedge e_{2j-1}$, $e_{2i-1}\wedge e_{2j}$, $e_{2i}\wedge e_{2j-1}$ and $e_{2i}\wedge e_{2j}$ represent totally $h$-singular
lines of $\PG(V_J)$; so, all of them are elements of
$\langle\HH_{J,2}\rangle$. In order to complete the proof we need to show that both $e_{2i-1}\wedge e_{2i}$ and
$e_{2j-1}\wedge e_{2j}$ lie in the span of $\HH_{J,2}$. To this purpose, take $\alpha,\beta\in\FF^*$ such that $\alpha\beta^{-1}\not\in\FF_0$ and
consider the four vectors
  \[ u_1^x=x e_{2i-1}+e_{2j-1},\qquad u_2^x=-x^{-\sigma}e_{2i}+e_{2j} \]
  with $x\in\{\alpha,\beta\}$. It is immediate to see that $u_1^x$ and $u_2^x$ are mutually orthogonal singular vectors. So, $\langle u_1^x\wedge u_2^x\rangle
  \in\HH_{J,2}$ and
\[ u_1^x\wedge u_2^x=-e_{2i-1}\wedge e_{2i}x^{1-\sigma}+e_{2j-1}\wedge e_{2j}+e_{2i-1}\wedge e_{2j}x+e_{2i}\wedge e_{2j-1}x^{-\sigma} \in
  \langle\HH_{J,2}\rangle. \]
Consequently,
  \begin{equation}\label{eH2J}
 u_1^{\beta}\wedge u_2^{\beta}-u_1^{\alpha}\wedge u_2^{\alpha}= (\alpha^{1-\sigma}-\beta^{1-\sigma})(e_{2i-1}\wedge e_{2i})+w\in\langle\HH_{J,2}\rangle,
  \end{equation}
with $w=(\beta-\alpha)e_{2i-1}\wedge e_{2j}+(\beta^{-\sigma}-\alpha^{-\sigma})e_{2i}\wedge e_{2j-1}\in\langle\HH_{J,2}\rangle$.
If $\beta^{1-\sigma}=\alpha^{1-\sigma}$, then $\alpha\beta^{-1}=(\alpha\beta^{-1})^{\sigma}\in\FF_0$ --- a contradiction.
So by~\eqref{eH2J}, $e_{2i-1}\wedge e_{2i}\in\langle\HH_{J,2}\rangle$. An analogous argument shows that also
$e_{2j-1}\wedge e_{2j}\in\langle \HH_{J,2}\rangle$ and this completes the proof.
\end{proof}
The following is shown by De Bruyn \cite[Corollary 1.2]{B10} (also Block and Cooperstein \cite[Corollary 3.2]{BC2012}). For completeness's sake, we provide a proof here.

\begin{lemma}\label{difetto zero}
Suppose $d_0 =0$. Then for all $k$ we have $\langle\HH_k\rangle=V_k$.
\end{lemma}
\begin{proof} Clearly, $\langle\HH_k\rangle\subseteq V_k$. To prove the reverse containment, we proceed by induction on $k$.
Recall from Section~\ref{preliminaries} that $V_k$ is spanned by the set
$E_k$ consisting of all $(e_{j_1}\wedge e_{j_2}\wedge\dots\wedge e_{j_k})$ where $\{j_1,\dots,j_k\}$ varies among all $k$-subsets of $\{1,\dots,2n\}$
and $j_1<j_2<\dots<j_k$.

If $k=1$, it is well known that the polar space $\HH_1$ generates $V$ and there is nothing to prove. Suppose the assertion holds for all values up to $k$ and consider $e:=e_{j_0,j_1,\dots,j_k}\in E_{k+1}$. We can assume without loss of generality $j_0\leq 2$. If $j_1>2$, let $J=\{3,4,\dots,N\}$. By induction, $\HH_{J,k}$ spans $V_{J,k}$. Since $\{1,2\}$ and $J$ are disjoint and $V_{\{1,2\}}$ is orthogonal to $V_J$, by Lemma~\ref{l11} we have
\[ e\in \langle e_{j_0}\rangle\wedge V_{J,k}= \langle e_{j_0}\rangle\wedge\langle\HH_{J,k}\rangle\subseteq
\langle\HH_{\{1,2\},1}\wedge\HH_{J,k}\rangle\subseteq\langle\HH_{\{1,2\}\cup J,k+1}\rangle\subseteq\langle\HH_{k+1}\rangle. \]
Suppose now $j_1=2$; consequently $j_0=1$ and $j_2>2$. Since, by hypothesis, $n\geq k+1$ and the indexes $j_2,j_3,\dots,j_k$ are
at most $k-1 (\leq n-2)$, there is at least one subset of the form $\{2i-1,2i\}$ with $i=2,3,\dots,n$ which is disjoint from $\{j_2,\dots,j_{k}\}$.

For simplicity of notation suppose $\{3,4\}$ to be such that $\{3,4\}\cap\{j_2,\dots,j_k\}=\emptyset$. Let $J=\{5,6,\dots,2n\}$.
By induction, $\langle \HH_{J,k-1}\rangle=V_{J,k-1}$. On the other hand, $e_1\wedge e_2\in V_{\{1,2,3,4\},2}=\langle\HH_{\{1,2,3,4\},2}\rangle$ by Lemma~\ref{l2}. So, $e\in\langle\HH_{\{1,2,3,4\},2}\rangle\wedge\langle\HH_{J,k-1}\rangle\subseteq\langle\HH_{k+1}\rangle$ by Lemma~\ref{l11}.
The lemma follows.
\end{proof}

\begin{theorem}\label{main-herm}
We have $\langle\HH_k\rangle=V_k$ for all $k$, independently of the anisotropic defect $d_0$ of $h$.
\end{theorem}
\begin{proof}
Clearly, $\langle\HH_k\rangle\subseteq V_k$. To prove the reverse containment we proceed by induction on $d_0$.
For $d_0=0$, the result is  given by Lemma~\ref{difetto zero}. Suppose $d_0>0$ and then argue by induction on $k$. For $k=1$ there is nothing to prove.
So assume $k>1$. We want to prove that for all $J$ with $|J|=k$ we have $e_J\in \langle \HH_k\rangle$.
Define $s:=k-|J\cap\{1,2,\dots,n\}|$. If $s=0$, then $e_J\in V_{\{1,\dots2n\},k}$. As $h_{\{1,\dots, 2n\}}$ is non-degenerate with anisotropic defect $0$, the result follows from Lemma~\ref{difetto zero} applied to $\HH_{\{1,2,\dots,2n\},k}\subseteq\HH_k$. Suppose $s>0$ and $J=\{j_1,j_2,\dots,j_{k-1},j\}$ with $j_1<j_2<\dots<j_{k-1}<j$ and
$j>2n$. We can assume without loss of generality $j=N$. Thus $e_J =e_{J\setminus\{N\}}\wedge e_N$.

Since $k-1<n$, there is at least one pair $X_i:=\{{2i-1,2i}\}$ with $X_i\cap J=\emptyset$ and $1\leq i\leq n$. Since the trace $\mathrm{Tr}:\FF\to\FF_0$ is surjective, there exists $t\in{\mathrm{Tr}^{-1}(-\kappa_N)}$.
 Then, the vector $u=e_{2i-1}+t e_{2i}+e_N$ is singular and $\langle u\rangle$ belongs to $\HH_{\{2i-1,2i,N\},1}$. On the other hand, also $\langle e_{2i-1}\rangle, \langle e_{2i}\rangle\in\HH_{\{2i-1,2i,N\},1}$. So $e_N\in\langle\HH_{\{2i-1,2i,N\},1}\rangle$.

Put $I_i := \{1, 2,\dots,N-1\}\setminus X_i$. Clearly, $e_{J\setminus\{N\}}\in V_{I_i, k-1}$. The form $h_{I_i}$ induced by $h$ on $V_{I_i}$  is non-degenerate with anisotropic defect $d_0-1$. Thus, by the inductive hypothesis (on $k$ or $d_0$, as we like) referred to $h_{I_i}$ we obtain that $e_{J\setminus\{N\}}\in \langle\HH_{I_i,k-1}\rangle$. So $e_J\in\langle\HH_{I_i,k-1}\rangle\wedge\langle\HH_{\{2i-1,2i,N\},1}\rangle$. However $V_{2i-1,2i,N}\perp V_{I_i}$. Hence, by Lemma~\ref{l11}, $\langle\HH_{I_i,k-1}\rangle\wedge\langle\HH_{\{2i-1,2i,N\},1}\rangle \subseteq \langle \HH_{I_i\cup\{2i-1, 2i, N\},k}\rangle = \langle \HH_k\rangle.$ Therefore $e_J\in \langle \HH_k\rangle$. The theorem follows.
\end{proof}

\begin{corollary}\label{dim hermitian grassmannians}
$\dim(\langle\HH_k\rangle) = {N\choose k}$.
\end{corollary}

Theorem~\ref{main-herm} and Corollary~\ref{dim hermitian grassmannians} yield Part~\ref{pt1} of Theorem~\ref{Main Theorem} in the non-degenerate case. As noticed at the beginning of this section, the complete statement of
Part~\ref{pt1} of Theorem~\ref{Main Theorem} follows by combining this partial result with Theorem~\ref{prel novo 3}.

\section{Symplectic $k$-Grassmannians }\label{Bsymp}

As recalled in Subsection~\ref{Survey}, Part~\ref{pt2} of Theorem~\ref{Main Theorem} holds in the non-degenerate case. By Theorem~\ref{prel novo 3}, it holds in the general case as well: $\dim(\langle\bS_k\rangle) = {N\choose k}-{N\choose{k-2}}$ provided that $k \leq n$, no matter how large the defect of the underlying alternating form can be.

Our goal in this section is to describe a generating set for $\langle\bS_k\rangle$ for $1\leq k \leq n$. This description will be crucial to prove the main result of Subsection~\ref{even char}.

Let $s\colon V\times V\to\FF$ be an alternating bilinear form with Witt index $n$ and singular defect $d =N-2n$, where $N = \dim(V)$. It is always possible to choose a basis $E=(e_1,e_2,\dots,e_N)$ of $V$ such that
\[ s(\sum_{i=1}^N x_ie_i,\sum_{i=1}^N y_ie_i)= \sum_{i=1}^{n}(x_{2i-1}y_{2i}-x_{2i}y_{2i-1}), \]
see~\cite[\S 5]{BA9}.
The subspace $\langle e_{2n+1},\dots, e_N\rangle$ is the radical of $s$. In the sequel, it will be convenient to keep a record of the form $s$ in our notation.  Thus, we write $\bS_k(s)$ instead of $\bS_k$. A basis
of $\langle \bS_k(s)\rangle $ when $s$ is non-degenerate (namely $d = 0$) is explicitly described by De Bruyn~\cite{B09} for arbitrary fields (see also
Premet and Suprunenko \cite{Premet} for fields of odd characteristic).
In this section we shall provide a generating set $E_k(s)$ for $\langle \bS_k(s) \rangle$ in the general case.

We first  introduce some notation. For $J = \{j_1, j_2,\dots, j_\ell\}\subseteq \{1, 2,\dots, N\}$ with $j_1 < j_2 <\dots < j_\ell$ define
\[e_J:=e_{j_1}\wedge e_{j_2}\wedge\dots\wedge e_{j_\ell}\]
with $e_\emptyset = 1$ by convention when $J = \emptyset$.

For $A \subseteq \{1, 2,\dots, n\}$, put $(2A-1):=\{ 2i-1 : i\in A\}$, $(2A):=\{ 2i : i\in A\}$ and define
\[e^+_A:=e_{(2A-1)} \quad \mbox{and} \quad e^-_A:=e_{(2A)}.\]
In particular, $e_{\emptyset}^+=e_{\emptyset}^-=1$ when $A = \emptyset$. For $1 \leq i < j \leq n$ let
\[u_i:=e_{2i-1}\wedge e_{2i} ~~ \mbox{and} ~~  u_{i,j}:=e_{2i-1}\wedge e_{2i}-e_{2j-1}\wedge e_{2j}.\]
Consider a set  $C=\{C_1,\dots,C_h\}$ where $C_1:=\{i_1,j_1\}$, $C_2:=\{i_2,j_2\},\dots,C_h=\{i_h,j_h\}$ are disjoint pairs
of elements of $\{1,2,\dots, n\}$, with the further assumptions $i_1< i_2<\dots< i_h$ and $i_r<j_r$ for $r= 1, 2,\dots, h$.
Let also $\overline{C}:=\{i_1,\dots,i_h,j_1,\dots,j_h\}$; clearly $|\overline{C}|=2h$. Given such a set $C$, let
\[u_C:=u_{i_1,j_1}\wedge u_{i_2,j_2}\wedge\dots\wedge u_{i_h,j_h} ~ ~\mbox{and}~~ u_{\emptyset}:=1.\]

\begin{setting}\label{condition}
In the sequel $(A,B,C,D)$ always stands for a quadruple with
\[A,B\subseteq\{1,\dots,n\}, ~ D\subseteq\{2n+1,2n+2,\dots,N\} ~~\mbox{and}~~ C=\{C_1,\dots,C_h\}\]
where $C_1:=\{i_1,j_1\}$, $C_2:=\{i_2,j_2\},\dots,C_h=\{i_h,j_h\}$ are disjoint pairs
of elements of $\{1,2,\dots,n\}$ such that $1\leq i_1< i_2<\dots< i_h$, $i_r<j_r\leq n$. Moreover, with $\overline{C}:=\{i_1,\dots,i_h,j_1,\dots,j_h\}$, we assume that $A,B,\overline{C}$ and $D$ are pairwise disjoint with $|A\cup B\cup \overline{C}\cup D|=k$.
\end{setting}

With $(A,B,C,D)$ as in Setting~\ref{condition}, put $e_{A,B,C,D}:=e_A^+\wedge e_B^-\wedge u_C\wedge e_D$. Define
\begin{equation}\label{base}
 E_k(s):=\{e_{A,B,C,D}\colon (A,B,C,D)\,\, \mbox{as in Setting } \ref{condition}\}.
\end{equation}
For short, denote by $e_{A,B,C}$ the factor $e^+_A\wedge e^-_B\wedge u_C$ of $e_{A,B,C,D} = (e^+_A\wedge e^-_B\wedge u_C)\wedge e_D$. Clearly, $e_{A,B,C} \in \bigwedge^{k-r}\overline{V}$, where $\overline{V} := \langle e_1, e_2,\dots, e_{2n}\rangle$ and $r = |D|$. The numbers $k-r = |A|+|B|+ 2|C| = |A|+|B|+|\overline{C}|$ and $r$ will be called the \emph{rank} and the
\emph{corank} of $e_{A,B,C}$, respectively.

Given $r \leq \min(d,k)$, the set of vectors $e_{A,B,C}$ of corank $r$ as defined above coincides with the set $E_{k-r}(\overline{s})$ defined as
in~\eqref{base}, but with $k-r$ instead of $k$ and $s$ replaced by its restriction $\overline{s}$ to $\overline{V}\times \overline{V}$. The form $\overline{s}$ is non-degenerate and $E_{k-r}(\overline{s})$ is a standard generating set for $\langle\bS_{k-r}(\overline{s})\rangle$ (see e.g. De Bruyn~\cite{B09}; also
\cite{Premet} in odd characteristic). Accordingly, the set
\[E_{k,r}(s) := \{e_{A,B,C,D}\in E_k(s) : |D| = r\} = E_{k-r}(\overline{s})\wedge\{e_D : D\subseteq \{2n+1,\dots, N\}, |D| = r\}\]
is a generating set for $\langle\bS_{k-r}(\overline{s})\rangle\wedge\bigwedge^rR$, where $R := \mathrm{Rad}(s) = \langle e_{2n+1},\dots, e_N\rangle$. By this fact,
Lemma~\ref{prel novo 1} and the fact that $E_k(s)$ is the disjoint union of the sets
$E_{k,0}(s), E_{k,1}(s),\dots, E_{k, m}(s)$ (where $m := \min(k,d)$), we immediately obtain the main result of this section:

\begin{lemma}\label{dim symplectic grassmannians}
If $k \leq n$ then $E_k(s)$ is a generating set for $\langle \bS_k(s)\rangle$.
\end{lemma}

\section{Orthogonal $k$-Grassmannians}\label{orthogonal}

In this section we shall prove Part~\ref{pt3} of Theorem~\ref{Main Theorem}. We firstly deal with the non-degenerate case. Having done that, a few words will be enough to fix the general case. We will treat separately the cases in which $\ch(\FF)$ is odd or even.

\subsection{The non-degenerate case in odd characteristic}\label{odd char}

Suppose $\ch(\FF)\neq 2$. Let $q:V\to\FF$ be a non-degenerate quadratic form of Witt index $n$ and anisotropic defect $\mathrm{def}_0(q) = d_0 = N-2n$. It is always possible to choose a basis $E=(e_1,\dots,e_{2n},e_{2n+1},\dots,e_N)$ of $V$ and $\kappa_{2n+1},\dots,\kappa_{N}\in\FF$  such that
\begin{equation}\label{quadform}
  q(\sum x_ie_i)=\sum_{i=1}^n x_{2i-1}x_{2i}+\sum_{j=2n+1}^N \kappa_j x_j^2,
\end{equation}
where each pair $(e_{2i-1},e_{2i})$ for $i=1,\dots,n$ is hyperbolic and the space  $\langle e_{2n+1},\dots,e_N\rangle$ is anisotropic in $\FF$, i.e.
\[ \sum_{j=2n+1}^N\kappa_jx_j^{2}=0 \Leftrightarrow x_{2n+1}=x_{2n+2}=\dots = x_N=0; \]
see~\cite[\S 6]{BA9}.
As in Section \ref{preliminaries}, $\cQ_k$ is the polar $k$-Grassmannian associated to $q$ and $\QQ_k = \varepsilon_k(\cQ_k)$ is its image by the Grassmann embedding. Given a subset $J \subset \{1, 2,\dots, N\}$, $q_J$ is the form induced by $q$ on $V_J$ and, for a positive integer $t$ not greater than the Witt index of $q_J$, $\cQ_{J,t}$ is the $t$-Grassmannian associated to $q_J$ and $\QQ_{J,t} = \varepsilon_t(\cQ_t)$. Let  $b:V\times V\to\FF$ be the bilinear form associated to $q$, i.e.
\[ b\left(\sum_{i=1}^N e_ix_i,\sum_{j=1}^N e_jy_j\right)= \sum_{i=1}^n(x_{2i-1}y_{2i}+x_{2i}y_{2i-1})+2\sum_{j=2n+1}^N \kappa_jx_{j}y_j.\]
The content of the following Lemma has been proved in~\cite{IP13} for $k<n$. In order to keep our treatment as self-contained as possible (and to clarify
that we do not take here the assumption $k<n$) we shall provide
a new and more elementary proof.

\begin{lemma}\label{pre-quad-odd}
Suppose $d_0=0$. Then $\langle\QQ_k\rangle=V_k$ for all $k \leq n$.
\end{lemma}
\begin{proof} Clearly, $\langle\QQ_k\rangle\subseteq V_k$. To prove $\langle\QQ_k\rangle\supseteq V_k$ we argue by induction on $k$. When $k=1$ it is well known that the polar space $\QQ_1$ spans $V$. We now show that $\langle\QQ_2\rangle\supseteq V_2$.
Indeed, we shall prove that for all $i,j$ with $1\leq i<j\leq n$ we have $e_i\wedge e_j\in\langle\QQ_2\rangle$. If $\{i,j\}\neq\{2x-1,2x\}$ for
 some $x\in\{1,\dots,n\}$, then the vectors $e_i$ and $e_j$  are $q$-singular and mutually orthogonal.
Hence  $\langle e_i, e_j\rangle$ is a $q$-singular line and so $\langle e_i\wedge e_j\rangle\in \QQ_2$. Suppose  $\{i,j\}=\{2x-1,2x\}$ for
 some $x\in\{1,\dots,n\}$, take $h\neq x$ and let
  \[ u_1=e_{2x-1}-e_{2h-1},\qquad
    u_2=e_{2x}+e_{2h}, \]
  \[ u_3=e_{2x-1}-e_{2h},\qquad
    u_4=e_{2x}+e_{2h-1}.\]
By construction $q(u_1)=q(u_2)=q(u_3)=q(u_4)=0$ and $b(u_1,u_2)=b(u_3,u_4)=0$ hence
$\langle u_1\wedge u_2\rangle,  \langle u_3\wedge u_4\rangle\in\QQ_{2}$. Furthermore,
  \[ u_1\wedge u_2+u_3\wedge u_4=2(e_{2x-1}\wedge e_{2x})+w+w' \]
with $w=e_{2x-1}\wedge e_{2h}-e_{2h-1}\wedge e_{2x}\in\langle\QQ_2\rangle$ and
$w'=e_{2x-1}\wedge e_{2h-1}-e_{2h}\wedge e_{2x}\in\langle\QQ_2\rangle$. So, $e_{2x-1}\wedge e_{2x}=e_{i}\wedge e_j\in\langle\QQ_2\rangle$.

Take now $k\geq 2$ and suppose that $\langle\QQ_k\rangle=V_k$; we claim $\langle\QQ_{k+1}\rangle=V_{k+1}$.
Let $J\subseteq\{1,\dots,2n\}$ with $|J|=k+1$. We show that $e_J\in\langle\QQ_{k+1}\rangle$. Since $k+1\leq n$, clearly $k<n$.
Take $i\in\{1,\dots,n\}$ such that $X_i\cap J\neq\emptyset$ where $X_i=\{2i-1,2i\}$. Let also $I_i :=\{1,\dots,2n\}\setminus X_i$. Put $t:=|X_i\cap J|$; clearly $t\in\{1,2\}$. By the inductive hypothesis on $k$, $e_{J\setminus X_i}\in\langle\QQ_{I_i,k-t+1}\rangle$ and $e_{J\cap X_i}\in\langle\QQ_{X_i,t}\rangle$. (Note that $q_{I_i}$ is non-degenerate with anisotropic defect $0$). So, by Lemma~\ref{l11}, being $X_i$  disjoint from $I_i$  and $V_{X_i}$ orthogonal to $V_{I_i}$, we have $e_J\in\langle\QQ_{X_i,t}\rangle\wedge\langle\QQ_{I_i,k-t+1}\rangle \subseteq\langle\QQ_{X_i\cup I_i, k+1}\rangle = \langle\QQ_{k+1}\rangle.$
\end{proof}

 \begin{theorem} \label{main-quad-odd}
We have $\langle\QQ_k\rangle=V_k$ for any value of the anisotropic defect $d_0$ of $q$ and any positive integer $k\leq n$.
\end{theorem}
\begin{proof}
Clearly, $\langle\QQ_k\rangle\subseteq V_k$. To prove $\langle\QQ_k\rangle\supseteq V_k$ we argue by induction on $d_0$.
The case $d_0=0$ is settled in Lemma~\ref{pre-quad-odd}. Suppose $d_0>0$.

Let $J\subseteq\{1,2,\dots, N\}$ with $|J|=k$. If $N\not\in J$, then $e_J\in V_{I,k}$ where $I := \{1, 2,\dots, N-1\}$. On the other hand, the form $q_I$ induced by $q$ on $V_I$ is non-degenerate with anisotropic defect $d_0-1$. So $\langle\QQ_{I,k}\rangle=V_{I,k}$ by the inductive hypothesis. Hence $e_J\in V_{I,k}=\langle\QQ_{I,k}\rangle\subseteq \langle\QQ_k\rangle$.

Suppose $N\in J$ and let $\kappa_N=q(e_N)$. So, $e_J = e_{J\setminus\{N\}}\wedge e_N$. Since $|J|\leq k$ and $N>2n$, we have $|J\setminus \{N\}|\leq k-1\leq n-1$. Hence there is necessarily at least one index $i\in\{1,\dots,n\}$ such that $\{2i-1, 2i\}\cap J=\emptyset$. For such a choice of $i$, let  $v_1:=e_{2i-1}-\kappa_N e_{2i}+e_N$ and  $v_2:=e_{2i-1}+\kappa_N e_{2i}-e_N$; so $\langle v_1\rangle, \langle v_2\rangle\in\QQ_{\{2i-1,2i,N\},1}$. Clearly, $v_1-v_2=-2\kappa_N e_{2i}+2e_N$. So $e_N\in\langle\QQ_{\{2i-1,2i,N\},1}\rangle$. On the other hand, $e_{J\setminus\{N\}}\in V_{I_i,k-1}$, where $I_i := \{1,2,\dots, N-1\}\setminus\{2i-1,2i\}$. The form $q_{I_i}$ is non-degenerate and has anisotropic defect $d_0-1$. Then $V_{I_i,k-1}= \langle\QQ_{I_i,k-1}\rangle$ by induction on $d_0$ (or on $k$). Consequently, $e_{J\setminus\{N\}}\in \langle\QQ_{I_i,k-1}\rangle$. It follows that  \[e_J = e_{J\setminus\{N\}}\wedge e_N \in \langle\QQ_{I_i,k-1}\rangle\wedge \langle\QQ_{\{2i-1,2i,N\},1}\rangle.\]
However $\langle\QQ_{I_i,k-1}\rangle\wedge \langle\QQ_{\{2i-1,2i,N\},1}\rangle\subseteq \langle\QQ_{I_i\cup\{2i-1,2i,N\},k}\rangle$ by Lemma~\ref{l11} and $\QQ_{I_i\cup\{2i-1,2i,N\},k} = \QQ_k$. Therefore $e_J\in\langle\QQ_k\rangle$. This completes the proof.
\end{proof}

\begin{corollary}\label{part 2. main theorem odd}
$\dim(\langle\QQ_k\rangle)={N\choose k}$.
\end{corollary}

\begin{remark}
  Theorem~\ref{main-quad-odd} includes Theorem~1.1 of~\cite{IP13} as the special case where $d = 1$, but the proof given in~\cite{IP13} is not as easy as  our proof of Theorem~\ref{main-quad-odd}. Note also that we have obtained
  Theorem~\ref{main-quad-odd} from Lemma~\ref{pre-quad-odd} while in~\cite{IP13} the analogue of our Lemma~\ref{pre-quad-odd} is obtained as a consequence of Theorem~1.1 of that paper.
\end{remark}

\subsection{The non-degenerate case in even characteristic} \label{even char}

Let now $\ch(\FF)=2$ and $q:V\to\FF$ be a non-degenerate quadratic form of Witt index $n$ and anisotropic defect $\mathrm{def}_0(q) = d_0 = N-2n$ as in Equation~\eqref{quadform}. As $\ch(\FF)=2$, the anisotropic part of the expression of $q$ can assume a more complex form than~\eqref{quadform}.
Actually, it is always  possible to determine a basis
\[E=(e_1,e_2,\dots,e_{2n},e_{2n+1},\dots,e_{2n+2m},e_{2n+2m+1},\dots,e_N)\]
of $V$ such that $q$ can be written as
\begin{equation}\label{eqeven}
  q\left(\sum_{i=1}^N x_ie_i\right)=q_0\left(\sum_{i=1}^{2n}x_ie_i\right)+  q_1\left(\sum_{i=2n+1}^{2n+2m}x_ie_i\right)+q_2\left(\sum_{i=2n+2m+1}^{N}x_ie_i\right),
\end{equation}
where
\[q_0\left(\sum_{i=1}^{2n}x_ie_i\right)=\sum_{i=1}^nx_{2i-1}x_{2i},\]
\[q_1\left(\sum_{i=2n+1}^{2n+2m}x_ie_i\right)=\sum_{i=1}^m(x_{2n+2i-1}x_{2n+2i} + \lambda_ix_{2n+2i-1}^2 + \mu_ix_{2n+2i}^2),\]
\[q_2\left(\sum_{i=2n+2m+1}^Nx_ie_i\right)= \sum_{j=2n+2m+1}^N\kappa_jx_j^2,\]
and $\lambda_i,\mu_i,\kappa_j\in \FF$ are such that the form
\[ q_{1,2}(\sum_{i=2n+1}^N x_{i}e_i):=q_1(\sum_{i=2n+1}^mx_ie_i)+q_2(\sum_{i=2n+2m}^Nx_ie_i) \]
defined on $V_{\{2n+1,\dots,N\}}$ is totally anisotropic; see, for instance, \cite[Proposition 7.31]{EKM08} and also \cite{BA9}.
In particular, for any $2n+2m+1\leq i,j\leq N$ each of the ratios $\kappa_i/\kappa_j$ must be a non-square in $\FF$ and the equations $\lambda_it^2+t+\mu_i=0$ must admit no solution in $\FF$. We will say that $q$ has parameters $[n,m,d'_0]$ where $n$ is its Witt index and $d'_0 :=N-2n-2m$. Clearly, $d_0 = 2m+d'_0$.

Denote by $f_q\colon V\times V\rightarrow \FF$, $f_q(x,y):=q(x+y)+q(x)+q(y)$ the alternating bilinear form polarizing  $q$. Explicitly, by Equation~\eqref{eqeven}, we have
\[ f_q(\sum_{i=1}^{N}x_ie_i,\sum_{i=1}^Ny_ie_i)=
  \sum_{i=1}^{n+m}(x_{2i-1}y_{2i}+x_{2i}y_{2i-1}). \]
Note that the Witt index of $f_q$ is $n+m+d'_0$ and $d'_0 = \dim(\mathrm{Rad}(f_q))$. So, $f_q$ is non-degenerate if and only if $d'_0 = 0$. Also, $n+m$ is the reduced Witt index of $f_0$.

Any totally singular $k$-space for $q$ is necessarily totally singular for $f_q$, but the converse does not hold.
This gives  $\cQ_k\subset\cS_k(f_q)$,  where by $\cS_k(f_q)$ we mean the symplectic $k$-Grassmannian related to
the form $f_q$. Consequently, $\langle\QQ_k\rangle\leq \langle \bS_k(f_q)\rangle$, where $\bS_k(f_q)$ is the image of $\cS_k(f_q)$ under the Grassmann embedding. In the remainder of this section we shall prove that actually $\langle\QQ_k\rangle=\langle \bS_k(f_q)\rangle$.

We shall stick to the notation of Sections~\ref{preliminaries} and~\ref{Bsymp}. In particular, $E_k(f_q)$ is the generating set for $\langle\bS_k\rangle$ defined in~\eqref{base}. We add the following to the notation of Section~\ref{Bsymp}.
Given $X\subseteq\{1,2\dots,n\}$ we write $[X]$ for $\{ 2x-1,2x \colon x\in X\}$.

The statement of the next lemma is proved in~\cite[Proposition 4.1(2)]{IP13} for $k < n$. We give an easier proof here, which works also for the case $k = n$.

\begin{lemma}\label{q-even-form}
Let $d_0 = 0$. Then $\langle\QQ_k\rangle=\langle \bS_k(f_q)\rangle$ for any $1\leq k \leq n$.
\end{lemma}
\begin{proof} As noticed above, $\langle\QQ_k\rangle\leq\langle\bS_k(f_q)\rangle$. We shall show that $\langle\QQ_k\rangle=\langle\bS_k(f_q)\rangle$.

Let $n=k=2$ and suppose first that $\FF=\FF_2$. A direct computation shows that the $6$ vectors
representing the lines of $\cQ_1=Q^+(3,2)$ span $\langle\bS_{2}(f_q)\rangle$, which has dimension $5$, and we are done.
As linear independence is  preserved taking field extensions, the $5$ vectors forming a basis of $\langle\QQ_2\rangle$ are linearly independent over any algebraic extension of $\FF$. So, $\langle\QQ_2\rangle=\langle\bS_2(f_q)\rangle$ for $n=2$.

We now show that the equality $\langle\QQ_k\rangle=\langle\bS_k(f_q)\rangle$ holds for any $n$ and any $2\leq k\leq n$. The generating set $E_k(f_q)$ for $\langle\bS_k(f_q)\rangle$ is formed by the vectors $e_{A,B,C,\emptyset}$. Observe first that for any $1\leq k\leq n$, all vectors of the form
$e_{A,B,\emptyset,\emptyset}$ with $|A\cup B|=k$ represent totally singular spaces for  the quadratic form $q$; so
  $e_{A,B,\emptyset,\emptyset}\in\langle\QQ_k\rangle$. Take now $e_{A,B,C,\emptyset}\in E_{k}(f_q)$ with $t:=|A\cup B|$, $0\leq t\leq k$. We have
  \[e_{A,B,C,\emptyset}=e_{A,B,\emptyset,\emptyset}\wedge e_{\emptyset,\emptyset,C,\emptyset}\]
with $e_{A,B,\emptyset,\emptyset}$ and $e_{\emptyset,\emptyset,C,\emptyset}$ as in Setting~\ref{condition} but with $k$ replaced by $t$ and $k-t$ respectively (note that $k-t$ is even). Then $e_{A,B,\emptyset,\emptyset}\in\langle\QQ_{(2A-1)\cup(2B),t}\rangle\leq
\langle\QQ_{[A\cup B],t}\rangle$ and $e_{\emptyset,\emptyset,C,\emptyset}\in \langle\bS_{[\overline{C}],k-t}\rangle$.

The vectors of the form $e_{\emptyset,\emptyset,{C},\emptyset}$ are in $\langle \QQ_{[\overline{C}],k-t}\rangle$ since, by Setting~\ref{condition}, $e_{\emptyset,\emptyset,C,\emptyset}=u_C=u_{i_1,j_1}\wedge u_{i_2,j_2}\wedge\dots\wedge u_{i_{k/2},j_{(k-t)/2}}$
 and each of the $u_{i_x,j_x}$ is in $\langle\bS_{\{2i_x-1,2i_x, 2j_x-1,2j_x\},2}\rangle$. Indeed, for each $x$ we have
 $\cQ_{\{2i_x-1,2i_x,2j_x-1,2j_x\},1}\cong Q^+(3,\FF)$, so, by what has been shown for $n=k=2$, we have
 $\langle\QQ_{\{2i_x-1,2i_x,2j_x-1,2j_x\},2}\rangle= \langle\bS_{\{2i_x-1,2i_x,2j_x-1,2j_x\},2}\rangle$ and
  \[ u_C\in \displaystyle\bigwedge_{x=1,\dots, (k-t)/2}\langle\QQ_{\{2i_x-1,2i_x,2j_x-1,2j_x\},2}\rangle=\langle \QQ_{[\overline{C}],k-t}\rangle. \]
  Hence $e_{A,B,C,\emptyset}\in \langle\QQ_{[A\cup B],t}\rangle\wedge\langle\QQ_{[\overline{C}],k-t}\rangle$. However, by
  Lemma~\ref{l11},
\[\langle\QQ_{[A\cup B],t}\rangle\wedge\langle\QQ_{[\overline{C}],k-t}\rangle\subseteq\langle\QQ_{[A\cup B]\cup\overline{C},k-t}\rangle\subseteq\langle\QQ_k\rangle.\]
Therefore  $e_{A,B,C,\emptyset}\in \langle\QQ_k\rangle$. It follows that $E_k(f_q)\subseteq \langle\QQ_k\rangle$. Consequently
$\langle\bS_k(f_q)\rangle\leq\langle\QQ_k\rangle$; this gives the thesis.
\end{proof}

\begin{lemma}\label{q-even-d'=0}
Let $d_0 =2m$, namely $d'_0=0$. Then $\langle\QQ_k\rangle=\langle \bS_k(f_q)\rangle$.
\end{lemma}
\begin{proof}
If $m=0$ then the thesis holds by  Lemma~\ref{q-even-form}. Suppose $m\geq 1$. We will show that any vector of $E_k(f_q)$ is contained in  $\langle\QQ_k\rangle$. We argue by induction on $k$. If $k=1$, it is well known that $\langle\QQ_1\rangle=V=\langle \bS_1(f_q)\rangle$ and we are done. Suppose now that the lemma holds for all values up to $k-1$. Let $(A,B,C,\emptyset)$ be as in Setting~\ref{condition}
with $|A\cup B\cup\overline{C}|=k-1$. Then, by the inductive hypothesis, $e_{A,B,C} := e_{A,B,C,\emptyset}$ belongs to $\langle\QQ_{k-1}\rangle$. Take $i\not\in A\cup B\cup\overline{C}$.

If $i\leq n$, since $q(e_{2i-1})=q(e_{2i})=0$, we have that $e_{A\cup\{i\},B,C}$ and $e_{A,B\cup\{i\},C}$ are in $\langle\QQ_k\rangle$.
Indeed $e_{A\cup\{i\},B,C}=e_{A,B,C}\wedge e_{2i-1}$ and $e_{A,B,C}\in\langle\QQ_{\{1,2,\dots,2i-2,2i+1,\dots,2n\},k-1}\rangle$,
since $i\not\in A\cup B$. So, $e_{A\cup\{i\},B,C} = e_{A,B,C}\wedge e_{2i-1}\in\langle\QQ_{[A\cup B\cup\overline{C}],k-1}\rangle\wedge\langle\QQ_{\{2i-1,2i\},1}\rangle$. However
 \[\langle\QQ_{[A\cup B\cup\overline{C}],k-1}\rangle\wedge\langle\QQ_{\{2i-1,2i\},1}\rangle\subseteq\langle\QQ_{[A\cup B\cup \overline{C}]\cup \{2i-1,2i\},k-1+1}\rangle\subseteq \langle\QQ_k\rangle\]
by Lemma~\ref{l11}. Hence $e_{A\cup\{i\},B,C}\in\langle\QQ_k\rangle$. The proof that $e_{A,B\cup\{i\},C}\in\langle\QQ_{k}\rangle$ is entirely analogous.

Suppose $i>n$. Since $|A\cup B\cup\overline{C}|=k-1\leq n-1$, there exists $j$ with $1\leq j\leq n$ and $j\not\in A\cup B\cup\overline{C}$. Let
$ v_1=e_{2j-1}+\lambda_{i-n}e_{2j}+e_{2i-1},$ where $\lambda_{i-n}=q(e_{2i-1})$. By construction, $\langle v_1\rangle\in\QQ_{\{2j-1,2j,2i-1,2i\},1}$; since $q(e_{2j-1})=q(e_{2j})=0$, also $\langle e_{2j-1}\rangle, \langle e_{2j}\rangle\in\QQ_{\{2j-1,2j,2i-1,2i\},1}$, thus forcing $e_{2i-1}\in\langle\QQ_{\{2j-1,2j,2i-1,2i\},1}\rangle$. Thus
\[ e_{A\cup\{i\},B,C}=e_{A,B,C}\wedge e_{2i-1}\in \langle\QQ_{[A\cup B\cup\overline{C}],k-1}\rangle\wedge
  \langle\QQ_{\{2j-1,2j,2i-1,2i\},1}\rangle. \]
By Lemma~\ref{l11}, $\langle\QQ_{[A\cup B\cup\overline{C}],k-1}\rangle\wedge\langle\QQ_{\{2j-1,2j,2i-1,2i\},1}\rangle\subseteq\langle\QQ_k\rangle$. Therefore $e_{A\cup\{i\},B,C}\in \langle\QQ_k\rangle$.
Likewise, let $ v_2=e_{2j-1}+\mu_{i-n}e_{2j}+e_{2i},$ where $\mu_{i-n}=q(e_{2i})$. The same argument as above shows that
\[ e_{A,B\cup\{i\},C}=e_{A,B,C}\wedge e_{2i}\in
  \langle\QQ_{[A\cup B\cup\overline{C}],k-1}\rangle\wedge
  \langle\QQ_{\{2j-1,2j,2i-1,2i\},1}\rangle\subseteq\langle\QQ_k\rangle. \]
Take now $(A,B,C,\emptyset)$ as in Setting~\ref{condition} with $|A\cup B\cup\overline{C}|=k-2$. By the inductive hypothesis, $e_{A,B,C}\in\langle\QQ_{k-2}\rangle$. We claim that $e_{A,B,C\cup\{\{i,h\}\}}\in\langle\QQ_k\rangle$ for any pair $\{i,h\}\subseteq \{1,2,\dots,n+m\}$ such that $\{i,h\}\cap(A\cup B\cup\overline{C})=\emptyset$.

Since $k-2\leq n-2$, there exist at least two distinct indexes $x,y$ with $1\leq x<y\leq n$ and
$x,y\not\in A\cup B\cup\overline{C}$. Also, there exists at least one index less or equal to $n$ which does not belong to $A\cup B\cup\overline{C}\cup \{i,h\}$ because $|(A\cup B\cup\overline{C}\cup \{i,h\})\cap \{1,2,\dots, n\}|\leq k-1\leq n-1$. Assume $x$ is that index.
We now distinguish two cases, according as $y$ coincides or not with $h$. Note that in the case $k<n$ it is always possible to find
  $x$ and $y$ distinct such that $\{i,h\}\cap\{x,y\}=\emptyset$ while if $k=n$ then we may have to take $y=h$.

\bigskip

\noindent
\begin{enumerate}[1.]
%  1.
\item Suppose $y\neq h$. Let $u_1=e_{2x-1}+\alpha e_{2x}+e_{2h-1}+e_{2i-1}$ and $u_2=e_{2y-1}+\beta e_{2y}+e_{2h}+e_{2i}$
with $\alpha=q(e_{2h-1}+e_{2i-1})$ and $\beta=q(e_{2h}+e_{2i})$ and $J:=\{2x-1,2x,2y-1,2y,2i-1,2i,2h-1,2h\}$.
 By construction, $f_q(u_1,u_2)=0$; so  $\langle u_1\wedge u_2\rangle\in\QQ_{J,2}$. On the other hand,
  \[ u_1\wedge u_2=(e_{2h-1}\wedge e_{2h}+e_{2i-1}\wedge e_{2i})+w \]
where
  \begin{multline*}
 w=(e_{2x-1}+\alpha e_{2x})\wedge(e_{2y-1}+\beta e_{2y})+ (e_{2x-1}+\alpha e_{2x})\wedge e_{2h}+ (e_{2y-1}+\beta e_{2y})\wedge e_{2h-1} + \\
    +(e_{2y-1}+\beta e_{2y}+e_{2h})\wedge e_{2i-1}+(e_{2x-1}+\alpha e_{2x}+e_{2h-1})\wedge e_{2i}
  \end{multline*}
is a linear combination of vectors of the form $e_{A,B,\emptyset}$. So $w\in\langle\QQ_{J,2}\rangle$ and, consequently,
$u_{h,i}=e_{2h-1}\wedge e_{2h}+e_{2i-1}\wedge e_{2i}\in\langle\QQ_{J,2}\rangle$. Since $J\cap(A\cup B\cup\overline{C})=\emptyset$ we have
  \[ e_{A,B,C\cup\{\{i,h\}\}}=e_{A,B,C}\wedge u_{h,i}\in
    \langle\QQ_{\{1,2,\dots,N\}\setminus J,k-2}\rangle\wedge
    \langle\QQ_{J,2}\rangle\subseteq\langle\QQ_{k}\rangle. \]
\item Suppose $y=h$. Let $u_1=\alpha e_{2x-1}+\beta e_{2x}+e_{2h-1}+e_{2i-1}$ and $u_2=e_{2x}+\alpha'e_{2h}+\beta'e_{2h-1}+e_{2i}$ with $\alpha,\alpha',\beta,\beta'$ such that $q(u_1)=q(u_2)=0$. This yields $\alpha\beta=q(e_{2i-1})$ and $\alpha'\beta'=q(e_{2i})$.
  Note $q(e_{2i-1})\not=0\not=q(e_{2i})$. We also want $f_q(u_1,u_2)=\alpha+\alpha'+1=0$. Take
  $\alpha'=\alpha+1$. Let now $J:=\{2x-1,2x,2h-1,2h,2i-1,2i\}$. Then, $\langle u_1\wedge u_2\rangle\in\QQ_{J,2}$ and
  \[ u_1\wedge u_2=\alpha(e_{2x-1}\wedge e_{2x})+\alpha'(e_{2h-1}\wedge e_{2h})+ (e_{2i-1}\wedge e_{2i})+ w \]
  where
  \begin{multline*}
    w=\alpha e_{2x-1}\wedge(\alpha'e_{2h}+\beta'e_{2h-1}+e_{2i})+\beta e_{2x}\wedge(\alpha' e_{2h}+\beta'e_{2h-1}+e_{2i})+ \\
    e_{2h-1}\wedge(e_{2x}+e_{2i})+e_{2i-1}\wedge(e_{2x}+e_{2h}+\beta'e_{2h-1})\in\langle\QQ_{J,2}\rangle,
  \end{multline*}
  since $w$ is a linear combination of vectors of the form $e_{A,B,\emptyset}$. On the other hand,
  \begin{multline*} u_1\wedge u_2-w=\alpha(e_{2x-1}\wedge e_{2x})+(1+\alpha)(e_{2h-1}\wedge e_{2h})+(e_{2i-1}\wedge e_{2i})= \\
    \alpha(\underbrace{e_{2x-1}\wedge e_{2x}+e_{2h-1}\wedge e_{2h}}_{v_1})+ (\underbrace{e_{2h-1}\wedge e_{2h}+e_{2i-1}\wedge e_{2i}}_{v_2}).
  \end{multline*}
Since $1\leq x,y\leq n$ and $h=y$, we have that $v_1\in \langle \bS_{\{2x-1,2x,2y-1,2y\},2}(f_q)\rangle$; however the quadratic form $q' = q_{\{2x-1,2x,2y-1,2y\}}$ is non-singular with anisotropic defect $0$. So, by Lemma~\ref{q-even-form},
  \[\langle \bS_{\{2x-1,2x,2y-1,2y\},2}(f_{q'})\rangle=\langle \QQ_{\{2x-1,2x,2y-1,2y\},2}\rangle\subseteq\langle\QQ_{J,2}\rangle.\]
  It follows that $v_2=u_1\wedge u_2-v_1-w\in\langle\QQ_{J,2}\rangle$.  By Lemma~\ref{l11} we now have
  \[ e_{A,B,C\cup\{\{i,h\}\}}=e_{A,B,C}\wedge v_2\in\langle{\QQ_{[A\cup B\cup\overline{C}],k-2}}\rangle\wedge
    \langle{\QQ_{J,2}}\rangle\subseteq\langle\QQ_{k}\rangle. \]
\end{enumerate}
  This completes the proof.
\end{proof}

\begin{theorem}\label{main-quad-even}
We have $\langle\QQ_k\rangle= \langle \bS_k(f_q)\rangle$ for any  choice of the parameters $[n,m,d'_0]$ of $q$ and any $k \leq n$.
\end{theorem}
\begin{proof}
When $d'_0=0$ the statement holds true by Lemma~\ref{q-even-d'=0}. Also, for $k=1$ it is well known that $\langle\QQ_1\rangle=V=\langle \bS_1(f_q)\rangle$. In this case there is nothing to prove.

Suppose $k>1,$ take $e_{A,B,C,D}\in E_k(f_q)$ and let $J=A\cup B\cup\overline{C}\cup D$. Clearly $|J|=k$. Let also
$h=k-|J\cap\{1,\dots,n,n+1,\dots,n+m\}|$. We argue by induction on $h$. If $h=0$, then
$e_{A,B,C,D}\in \langle \bS_{\{1,\dots,2n+2m\},k}(f_q)\rangle=\langle\QQ_{\{1,2,\dots,2n+2m\},k}\rangle$
  by Lemma~\ref{q-even-d'=0}, since the polar space $\cQ_{\{1,2,\dots,2n+2m\}}$ has $d'_0=0$ and we are done.
  Suppose $h>1$; in particular $D\neq\emptyset$ and take $j\in D$. As $D \subseteq \{2n+2m+1,\dots, N\}$ (see Setting~\ref{condition}), it must be $j>2n+2m$.  Moreover, there necessarily exists  an index $1\leq i\leq n$ such that $\{2i-1,2i\}\cap J=\emptyset$. Let $u=e_{2i-1}+\kappa_je_{2i}+e_{j}$ where $q(e_j)=\kappa_j$. Clearly $q(u)=0$, so $\langle u\rangle\in\QQ_{\{2i-1,2i,j\},1}$. Since
$\langle e_{2i-1}\rangle, \langle e_{2i}\rangle\in\QQ_{\{2i-1,2i,j\}}$ we have $e_j\in\langle\QQ_{\{2i-1,2i,j\},1}\rangle$.
Now let $D'=D\setminus\{j\}$ and $J'=A\cup B\cup\overline{C}\cup D'$.Then,
$h':=(k-1)-|J'\cap\{1,2,\dots,n+m\}|=(k-1)-|J\cap\{1,2,\dots,n+m\}|=h-1$. Moreover, if $I_{ij} := \{1,2,\dots,N\}\setminus\{2i-1,2i,j\}$, the form $q_{I_{ij}}$ is non-degenerate with parameters $[n-1, m, d'_0-1]$. So $e_{A,B,C,D'}\in\langle\QQ_{I_{ij},k-1}\rangle$ by the inductive hypothesis on $h$ (but induction on $k$ or on $d'_0$ would work as well). By Lemma~\ref{l11},
\[ e_{A,B,C,D}=e_{A,B,C,D'}\wedge e_j\in\langle\QQ_{I_{ij},k-1}\rangle\wedge \langle\QQ_{\{2i-1,2i,j\},1}\rangle\subseteq\langle\QQ_k\rangle. \]
This completes the proof.
\end{proof}

\begin{corollary}\label{part 2 main theorem}
$\dim(\langle\QQ_k\rangle)={N\choose k}-{N\choose{k-2}}$.
\end{corollary}

\begin{remark}
  Theorem \ref{main-quad-even} includes Theorem~1.2 of~\cite{IP13} as the special case where $d = 1$, but the proof given in~\cite{IP13} is far more elaborate than our proof of Theorem~\ref{main-quad-even}. Note also that we have obtained Theorem~\ref{main-quad-even} from Lemmas~\ref{q-even-form}
  and~\ref{q-even-d'=0} while in~\cite{IP13} the analogue of our
  Lemma~\ref{q-even-form} is obtained as a consequence of Theorem~1.2
  of that paper.
\end{remark}

\subsection{End of the proof of Part~\ref{pt3} of Theorem~\ref{Main Theorem}}

Let now $q$ be degenerate with reduced Witt index $n$ and let $k \leq n$. When $\ch(\FF) \neq 2$ the equality $\dim(\langle\QQ_k\rangle) = {N\choose k}$ follows from Corollary~\ref{part 2. main theorem odd} and Part~\ref{pt1a} of Theorem~\ref{prel novo 3}. In this case $\langle\QQ_k\rangle = V_k$.

Let $\ch(\FF) = 2$. Then $\dim(\langle\QQ_k\rangle) = {N\choose k}-{N\choose{k-2}}$ by Corollary~\ref{part 2 main theorem} and Part~\ref{pt2a} of Theorem~\ref{prel novo 3}. Let $f_q$ be the bilinearization of $q$. Then $\langle\QQ_k\rangle \leq \langle\bS_k(f_q)\rangle$. Moreover $\dim(\langle\bS_k(f_q)\rangle) = {N\choose k}-{N\choose{k-2}}$ by
Part~\ref{pt2} of Theorem~\ref{Main Theorem}, as already proved in Section~\ref{Bsymp}. The equality
$\langle\QQ_k\rangle = \langle\bS_k(f_q)\rangle$ follows.

The proof of Theorem~\ref{Main Theorem} is complete.

\section{The case $n < k \leq n+d$}\label{nova Sec}

We shall now prove Corollary~\ref{Main Corollary}. Let $n < k \leq n+d$. Note firstly that the formula of Corollary~\ref{prel novo 2} also holds for $n < k \leq n+d$ provided that the summation index $i$ is subject to the restriction $k-i \leq n$, namely $i\geq k-n$. Similarly for the statement  of
Lemma~\ref{prel novo 1}. Thus,
\begin{equation}\label{prel novo 1 2}
\left.\begin{array}{rcl}
\langle\PP_k\rangle& = & \displaystyle \bigoplus_{i= k-n}^{\mathrm{min}(d,k)}\langle\overline{\PP}_{k-i}\rangle\wedge\bigwedge^iR,\\
 & & \\
\dim(\langle\PP_k\rangle)& = & \displaystyle \sum_{i= k-n}^{\mathrm{min}(d,k)}\dim(\langle\overline{\PP}_{k-i}\rangle)\cdot{d\choose i}.
\end{array}\right\}
\end{equation}
The dimensions of the spaces $\langle\overline{\PP}_{k-i}\rangle$ are known by Theorem~\ref{Main Theorem} (recall that the form $\overline{\eta}$ induced by $\eta$ on $\overline{V}$ is non-degenerate): explicitly,  $\dim(\langle\overline{\PP}_{k-i}\rangle)$ is equal to ${{N-d}\choose{k-i}}$ or ${{N-d}\choose{k-i}}-{{N-d}\choose{k-i-2}}$ according to the type of $\eta$ and the characteristic of $\FF$ (when $\eta$ is quadratic). By putting these values in the second equation of~\eqref{prel novo 1 2}
and using Lemma~\ref{lemma-sum}
we obtain the formulas of Corollary~\ref{Main Corollary}. 

The last claim of Corollary~\ref{Main Corollary}, namely $\langle\PP_k\rangle \subset V_k$, is clear. Indeed $\sum_{i=0}^{k-n-1}{{N-d}\choose{k-i}}{d\choose i} > 0$. So, $\dim(\langle\PP_k\rangle) < {N\choose k} = \dim(V_k)$ in Case~\ref{cc1} of Corollary~\ref{Main Corollary}. In Case~\ref{cc2} we have $\langle\overline{\PP}_{k-i}\rangle \subset \overline{V}_{k-i}$ for every admissible value of $i$. Consequently $\langle\PP_k\rangle \subset V_k$ by the first equation of~\eqref{prel novo 1 2}.   

\section{Generalizing the Weyl embedding }\label{Weyl generalizzato}

In this section we will show that it is possible to define a projective embedding $\widetilde{\varepsilon}_0$ of a subgeometry $\Gamma_0$ of a given geometry $\Gamma$ (in general $\Gamma_0$ is defined over a subfield) starting from a given projective embedding $\varepsilon_0$ of $\Gamma_0$ and knowing that there exist two projective embeddings $\varepsilon$ and $\widetilde{\varepsilon}$ of $\Gamma$ such that $\varepsilon$ induces $\varepsilon_0$ on $\Gamma_0$ and is a quotient of $\widetilde{\varepsilon}$. We will then apply this result to obtain a generalization of the Weyl embedding
to orthogonal Grassmannians in even characteristic defined by non-degenerate quadratic forms with anisotropic defect $d_0 > 1$ but $d'_0 \leq 1$.

In the Introduction of this paper we have used the word ``embedding'' in a somewhat loose way, avoiding a strict definition but thinking of an embedding $\varepsilon:\Gamma\rightarrow\Sigma$ of a point-line geometry $\Gamma$ in a projective space $\Sigma$ as an injective mapping from the point-set $P$ of $\Gamma$ to the point-set of $\Sigma$, neither requiring that $\varepsilon$ is `projective', namely it maps lines of $\Gamma$ onto lines of $\Sigma$ (although this requirement is satisfied in most of the cases we consider), nor that $\varepsilon(P)$ spans $\Sigma$, even when $\varepsilon$ is indeed `projective'. This free way of talking was the right one in that context, since we focused on the problem of determining the span of the set of points $\varepsilon(P)$ (namely $\varepsilon_k({\cP}_k)$) inside $\Sigma$ (namely $\PG(\bigwedge^kV)$). However, in the sequel, sticking to that
lax setting would cause some troubles. Thus, from now on, we shall adopt a sharper terminology. Following Shult~\cite{SH95}, given a point-line geometry $\Gamma$ where the lines are sets of points, we say that an injective mapping from the point-set $P$ of $\Gamma$ to the point-set of a projective geometry $\PG(V)$ is a \emph{projective embedding} of $\Gamma$ in $\mathrm{PG}(V)$ if:
\begin{enumerate}[(E1)]
\item\label{eE1} for every line $\ell$ of $\Gamma$, the set $\{\varepsilon(p) : p\in \ell\}$ is a line of $\mathrm{PG}(V)$;
\item\label{eE2} the set $\varepsilon(P)$ spans $\mathrm{PG}(V)$.
\end{enumerate}
When talking about projective embeddings in the Introduction of this paper we assumed (E\ref{eE1}) but not (E\ref{eE2}), but now we also require (E\ref{eE2}). According to (E\ref{eE2}), the dimension $\dim(\varepsilon)$ of $\varepsilon$ is just $\dim(V)$.

All embeddings to be considered henceforth  are projective in the sense we have now fixed. In the sequel we shall also deal with morphisms of embeddings. We refer to Shult~\cite{SH95} for this notion.

\subsection{Liftings of projective embeddings}\label{s:lift}

Let $\Gamma$ be a point-line geometry and consider two projective embeddings
\[ \widetilde{\varepsilon}:\Gamma\to\PG(\widetilde{V}),\qquad \varepsilon:\Gamma\to\PG(V) \]
defined over the same field $\FF$ such that $\varepsilon$ is
a quotient of $\widetilde{\varepsilon}$. Let also
$\varphi:\widetilde{\varepsilon}\to\varepsilon$ be the
morphism from $\widetilde{\varepsilon}$ to $\varepsilon$, i.e. $\varphi\colon \widetilde{V}\to V$ is a semilinear mapping such that $\varphi\circ \widetilde{\varepsilon}=\varepsilon$. Note that $\varphi$  is uniquely determined up to scalars. In particular, $\PG(\widetilde{V}/K)\cong\PG(V)$ where $K:=\ker(\varphi)$. Note also that $\varphi$ induces a bijection from $\widetilde{\varepsilon}(\Gamma)$ to $\varepsilon(\Gamma)$ (in fact an isomorphism of point-line geometries). Equivalently, $K$ contains no point $\widetilde{\varepsilon}(p)\in \widetilde{\varepsilon}(\Gamma)$ and $\langle \widetilde{\varepsilon}(p), \widetilde{\varepsilon}(q)\rangle\cap K = 0$ for any two points $p$ and $q$ of $\Gamma$.

\begin{lemma}\label{uniq}
  For any point $p$ of $\Gamma$ and any vector $v\in\varepsilon(p)$ with
  $v\neq0$, there is exactly one vector
  $\widetilde{v}\in\widetilde{\varepsilon}(p)$ such that
  $\varphi(\widetilde{v})=v$.
\end{lemma}
\begin{proof}
We warn the reader that in the following we will often regard the point $\varepsilon(p)$ as a $1$-dimensional vector subspace.

Since $\varphi$ is surjective, $\varphi^{-1}(\varepsilon(p))\not=\emptyset$. Hence, since by hypothesis $v\in\varepsilon(p)$
(i.e. $\langle v\rangle=\varepsilon (p)$), there is $\widetilde{x}\in\widetilde{V}$ such that $\varphi^{-1}({v})=\widetilde{x}+K$ and
$\widetilde{\varepsilon}(p)=\langle \widetilde{x} \rangle$. Hence $|(\widetilde{x}+K)\cap\widetilde{\varepsilon}(p)|\geq 1$.
If $|(\widetilde{x}+K)\cap\widetilde{\varepsilon}(p)|\geq 2$ then $\widetilde{\varepsilon}(p)\cap K\neq 0$ and, since $\dim(\widetilde{\varepsilon}(p))=1$,
 we would have  $\widetilde{\varepsilon}(p)\subseteq K$ --- a contradiction. So, $|(\widetilde{x}+K)\cap\widetilde{\varepsilon}(p)|=1$.
  Take $\widetilde{v}\in\widetilde{\varepsilon}(p)$. Then $\langle\varphi(\widetilde{v})\rangle=\varepsilon(p)=\langle v\rangle$;
  so $\varphi(\widetilde{v})=tv$ for some $t\neq 0$. Up to replacing $\widetilde{v}$ with $t^{-1}\widetilde{v}$, we can suppose $\varphi(\widetilde{v})=v$, so $\widetilde{v}\in(\widetilde{x}+K)\cap\widetilde{\varepsilon}(p)$. Consequently,  $\varphi^{-1}(v)\cap\widetilde{\varepsilon}(p)=\{\widetilde{v}\}$.
\end{proof}

\begin{definition}
With $\widetilde{v}$ as in Lemma~\ref{uniq}, we call $\widetilde{v}$ the \emph{lifting} of $v$ to $\widetilde{V}$ through $\varphi$ and write $\widetilde{v}=\varphi^{-1}(v)$.
\end{definition}

Let now $\Gamma_0$ be a subgeometry of $\Gamma$ defined over a subfield $\FF_0$ of $\FF$. Suppose that the vector space $V$ admits a basis $E$
such that the restriction of $\varepsilon$ to $\Gamma_0$ is the natural field extension of a projective embedding $\varepsilon_0:\Gamma_0\to\PG(V_0)$
 where $V_0$ is the span of $E$ over $\FF_0$.

In order to avoid unnecessary complications, we assume that $\varphi$ is linear. This hypothesis suits our needs in Section~\ref{Wg}. Moreover, it is not as restrictive as it can look (see below, Remark~\ref{R65}).

Let $\varphi^{-1}(\varepsilon_0(\Gamma_0))$ be the set of all liftings to $\widetilde{V}$ of vectors of $V$  representing points of the form
$\varepsilon_0(p)$  with $p\in\Gamma_0$ and let $\widetilde{V}_0$ be the span of $\varphi^{-1}(\varepsilon_0(\Gamma_0))$
over $\FF_0$. By Lemma~\ref{uniq}, every vector $v_0\in\varepsilon_0(p)$ with $p\in\Gamma_0$ admits a unique lifting
$\widetilde{v}_0\in\widetilde{\varepsilon}(\Gamma)$. Furthermore, since $\varphi$ is $\FF$--linear (whence also $\FF_0$--linear)
it is immediate to see that the set of the liftings of the non-zero vectors of $\varepsilon_0(p)$ is the $\FF_0$-span $\langle \widetilde{v}_0\rangle_{\FF_0}$
of the lifting $\widetilde{v}_0$ of any non-zero vector $v_0\in \varepsilon_0(p)$. So, the following definition is well posed.

\begin{definition}
Let $\widetilde{\varepsilon}_0:\Gamma_0\to\PG(\widetilde{V}_0)$ be the map defined by the clause $\widetilde{\varepsilon}_0(p):=\langle\widetilde{v}_0\rangle_{\FF_0}$ for $p$ a point of $\Gamma_0$, $\varepsilon_0(p)=\langle v_0\rangle_{\FF_0}$ and $\widetilde{v}_0:=\varphi^{-1}(v_0)$. We call $\widetilde{\varepsilon}_0$ the \emph{lifting} of $\varepsilon_0$ to $\widetilde{V}_0$ through $\varphi$.
\end{definition}

\begin{theorem}\label{mainlift}
The lifting $\widetilde{\varepsilon}_0$ of $\varepsilon_0$ is a projective embedding of $\Gamma_0$ in $\PG(\widetilde{V}_0)$ and $\varphi$
induces a linear morphism $\varphi_0:\widetilde{\varepsilon}_0\to\varepsilon_0$.
\end{theorem}
\begin{proof}
 We first show that the image of any three collinear points $r,s,t\in\Gamma_0$ is contained in a line of $\PG(\widetilde{V}_0)$.
Let $r'\in\varepsilon_0(r)$, $s'\in\varepsilon_0(s)$, $t'\in\varepsilon_0(t)$, with $r',s',t'\neq 0$. Since $\varepsilon_0$ is a projective embedding,
there exist $x,y\in\FF_0$ such that $r'=xs'+yt'$. By Lemma~\ref{uniq} the vectors $\widetilde{r}'$, $\widetilde{s}'$, $\widetilde{t}'$
such that $\varphi(\widetilde{r}')=r'$, $\varphi(\widetilde{s}')=s'$, $\varphi(\widetilde{t}')=t'$ are uniquely determined.
Clearly $\widetilde{\varepsilon}_0( r)=\langle \widetilde{r}'\rangle$, $\widetilde{\varepsilon}_0(s)=\langle \widetilde{s}'\rangle$,
$\widetilde{\varepsilon}_0(t)=\langle \widetilde{t}'\rangle$. Since $\varphi$ is linear we have also
\[ \varphi(x \widetilde{s}'+y \widetilde{t}')= x\varphi(\widetilde{s}')+y\varphi(\widetilde{t}')=xs'+yt'=r'. \]
So, by Lemma~\ref{uniq} we have $x\widetilde{s}'+y\widetilde{t}'=\widetilde{r}'$ with $x,y\in\FF_0$. Thus,
$\widetilde{r}'\in\langle\widetilde{s}',\widetilde{t}'\rangle_{\FF_0}$.

Since $\varepsilon_0$ is projective, for each value of $x,y\in\FF_0$, the vector $xs'+yt'$ represents a point $r\in\Gamma_0$, so
also $x\widetilde{s}'+y\widetilde{t}'$ represents a point of $\Gamma_0$. This proves that the image of a line of
$\Gamma_0$ by means of $\widetilde{\varepsilon}_0$ is a (full) line in $\PG(\widetilde{V}_0)$.

Define $\varphi_0$ by  $\varphi_0(\widetilde{v}_0)=v_0$ for any $\widetilde{v}_0\in\widetilde{\varepsilon}_0(p)$ and $p\in\Gamma_0$
and extend the function by linearity to all $\widetilde{V}_0$. So, $\varphi_0(\widetilde{\varepsilon}_0(p))=\langle v_0\rangle_{\FF_0}$ where
${\varepsilon}_0(p)=\langle{v}_0\rangle_{\FF_0}$ and $\varphi_0:\widetilde{V}_0\to V_0$ is a restricted truncation of the $\FF$--linear map
$\varphi:\widetilde{V}\to V$. In particular, $\varphi_0$ is $\FF_0$--linear. This completes the proof.
\end{proof}

\begin{remark}\label{nova}
Clearly, $\ker(\varphi_0) = \widetilde{V}_0\cap\ker(\varphi)$. So, if $\varphi$ is an isomorphism then $\varphi_0$ too is an isomorphism.
\end{remark}

\begin{remark}
It follows from Theorem~\ref{mainlift} that $\dim(\widetilde{V}_0)\geq \dim(V)$. On the other hand, it can happen that $\widetilde{V}_0$ is not contained in the $\FF_0$-span of any basis of $\widetilde{V}$. If that is the case, then it might happen that $\dim(\widetilde{V}_0) > \dim(\widetilde{V})$.
\end{remark}

\begin{remark}\label{R65}
  We have assumed that $\varphi$ is linear but what we have said remains valid if $\varphi$ is semi-linear without being linear, with the following unique modification: if $\sigma$ is the automorphism of $\FF$ associated to $\varphi$, then we must define $\widetilde{V}_0$ as the $\FF'_0$-span of
$\varphi^{-1}(\varepsilon_0(\Gamma_0))$ with $\FF'_0 := \sigma^{-1}(\FF_0)$.
\end{remark}

\subsection{Weyl-like embeddings}\label{Wg}

Consider an orthogonal $k$-Grassmannian $\cQ_k$ defined by a non-degenerate quadratic form of Witt index $n$ and anisotropic defect $d_0$. Suppose $k<n$ and $d_0\leq 1$. Then, as explained in Section~\ref{Main Results Sec}, the geometry $\cQ_k$ affords both the Weyl embedding and the Grassmann embedding. These are both projective. If either $\ch(\FF)\not=2$ or $k = 1$ then the Weyl embedding and the Grassmann embedding are essentially the same while if $\ch(\FF)=2$ and $k > 1$ then the Grassmann embedding is a proper quotient of the Weyl embedding.

If $d_0 >1$ then the Weyl embedding cannot be considered. Nevertheless we shall
manage to define a generalization of the Weyl embedding for orthogonal Grassmannians defined by quadratic forms with defect greater than $1$. To this aim, we need a couple of lemmas on algebraic extensions of the underlying field $\FF$ of a quadratic form $q$.

\begin{lemma}\label{l:qc}
Let $\ch(\FF) = 2$ and let $q:V\to \FF$ be the generic non-degenerate quadratic form with parameters $[n,m,d'_0]$ as given by~\eqref{eqeven}. Then there exists a field extension $\widehat{\FF}$ such that the extension $\widehat{q}$ of $q$ to $\widehat{V}:=V\otimes\widehat{\FF}$ admits the following representation with respect to a suitable basis $\widehat{E}$ or $\widehat{V}$:
    \begin{equation}
      \label{oqe}
      \widehat{q}(x_1,\dots, x_N) ~ = ~
      \begin{cases}
        \displaystyle\sum_{i=1}^{n+m}x_{2i-1}x_{2i} & \mbox{if} ~ 2n+2m = N,\\[.6cm]
        \displaystyle \sum_{i=1}^{n+m}x_{2i-1}x_{2i} + \kappa_{2n+2m+1}x_{2n+2m+1}^2 & \mbox{if} ~ 2n+2m < N.
                \end{cases}
            \end{equation}
Moreover, if we require that $\widehat{\FF}$ has minimal degree over $\FF$, then $\widehat{\FF}$ is a uniquely determined algebraic extension of $\FF$ of degree $|\widehat{\FF}:\FF| \leq 2m+2d''_0$ where $d''_0 := \max(0, d'_0-1)$.
\end{lemma}
\begin{proof}
If $2n+2m < N$, put $\kappa := \kappa_{2n+2m+1}$, for short. Consider an extension $\widehat{\FF}$ containing the two roots of each equation $\lambda_it^2+t+\mu_i=0$ for $i=1,\dots,m$, say $\alpha_i$ and $\beta_i$, and also all elements $\delta_j = (\kappa_j/\kappa)^{1/2}\in{\FF}$ for all $j=2n+2m+2,\dots,N$. Define $\gamma_i:=(\alpha_i+\beta_i)^{1/2}$. Since $\ch(\FF)=2$, we have $\gamma_i\in\widehat{\FF}$. Now let
\[ \left.\begin{array}{rcl} \widehat{e}_{2n+2i-1}&=&e_{2n+ri-1}\alpha_i\gamma_i^{-1}+
                                                      e_{2n+2i}\gamma_i^{-1} \\
           \widehat{e}_{2n+2i}&=&e_{2n+2i-1}\beta_i\gamma_i^{-1}+ e_{2n+2i}\gamma_i^{-1} \\
         \end{array}\right\} \mbox{ for } i=1,2,\dots,m. \]
Then,
         \[ \widehat{q}(\widehat{e}_{2n+2i-1})=\lambda_i\alpha_i^2\gamma_i^{-2}+\alpha_i\gamma_i^{-2}+\mu_i\gamma_i^{-2}=\gamma_i^{-2}(\lambda_i\alpha_i^2+\alpha_i+\mu_i)=0, \]
       \[\widehat{q}(\widehat{e}_{2n+2i})=\lambda_i\beta_i^2\gamma_i^{-2}+\beta_i\gamma_i^{-2}+\mu_i\gamma_i^{-2}=\gamma_i^{-2}(\lambda_i\beta_i^2+\beta_i+\mu_i)=0, \]
\[f_{\widehat{q}}(\widehat{e}_{2n+2i-1},\widehat{e}_{2n+2i}) \gamma_i^{-2}(\alpha_i+\beta_i)=1.\]
When $2n+2m<N$, put also $\widehat{e}_{2m+2n+1}:=e_{2m+2n+1}$ and
       \[ \widehat{e}_j:=e_{j-1}\delta_{j-1}^{-1}+e_j\delta_j^{-1},\quad \mbox{for } j=2n+2m+2,\dots,N,\]
with $\delta_{j-1}:= 1$ when $j = 2n+2m+2$. Clearly, $\widehat{q}(\widehat{e}_{2n+2m+1})=\kappa$ and, for $j=2n+2m+2,\dots,N$,
       \[ \widehat{q}(\widehat{e}_j)=\kappa_{j-1}\delta_{j-1}^2+\kappa_j\delta_j^{-2}=\kappa +\kappa=0,\quad\mbox{for } j = 2n+2m+2,\dots,N \]
\[ f_{\widehat{q}}(\widehat{e}_i, \widehat{e}_j)=0, \quad\mbox{for } i, j = 2n+2m+1,\dots,N.\]
(As for the latter equality, recall that $f_q(e_i,e_j) = 0$ for any $i, j = 2n+2m+1,\dots, N$).  So, for $j \geq 2n+2m+1$ the vector $\widehat{e}_j$ is orthogonal to all vectors of $\widehat{V}$, namely it belongs to the radical of $\widehat{q}$. Let now
\begin{equation}\label{base nova}
\widehat{E}:=(e_1,\dots,e_{2n},\widehat{e}_{2n+1},\dots, \widehat{e}_{2n+2m},\widehat{e}_{2n+2m+1},\dots,\widehat{e}_N).
\end{equation}
This is a basis of $\widehat{V}$. With respect to this basis, $\widehat{q}$ assumes the form \eqref{oqe}.

The last claim of the Lemma remains to be proved. Suppose that $\widehat{\FF}$ has minimal degree over $\FF$. Then $\widehat{\FF}$ is an algebraic extension of $\FF$ obtained from $\FF$ as a series of quadratic extensions by adding roots of the equations $\lambda_it^2+t+\mu_i = 0$ for $i = 1, 2,\dots, m$ and $t^2 = \kappa_j/\kappa$ for $j = 2n+2m+2,\dots, N$. The degree of such an extension is at most $2m+2d''_0$ with $d''_0 = d'_0-1$ if $d'_0 > 0$, otherwise $d''_0 = 0$.

Still assuming that $\widehat{\FF}$ has minimal degree, let $\widehat{\FF}'$ be another extension of $\FF$ such that the extension $\widehat{q}'$ of $q$ to $\widehat{V}' = V\otimes\widehat{\FF}'$ admits the representation \eqref{oqe}. If some of the polynomials $\lambda_it^2+t+\mu_i$ or $t^2 + \kappa_j/\kappa$ are
irreducible in $\widehat{\FF}'$, then $\widehat{V}'$ admits $2$-subspaces $X$ totally anisotropic for $\widehat{q}'$ and such that $X+X^\perp = \widetilde{V}'$. However this situation does not fit with \eqref{oqe}.
Therefore all polynomials $\lambda_it^2+t+\mu_i$ for $i = 1, 2,\dots, m$ are reducible in $\widetilde{\FF}'$ as well as all polynomials $t^2 + \kappa_j/\kappa$ for $j = 2n+2m+2,\dots, N$. It follows that $\widehat{\FF}'$ contains $\widehat{\FF}$ (modulo isomorphisms, of course).
\end{proof}

Clearly, the form $\widehat{q}$ of Lemma \ref{l:qc} is non-degenerate if and only if $d'_0\leq 1$. If this is the case then $d'_0 = \mathrm{def}_0(\widehat{q})$. On the other hand, if $d'_0>1$ then the radical of ${q}$ has dimension $d'_0-1$  and is generated by the last $d'_0-1$ vectors $\widehat{e}_{2m+2n+2},\dots,\widehat{e}_{N}$ of the basis $\widehat{E}$ defined in \eqref{base nova}.

In Lemma \ref{l:qc} we have assumed that $\ch(\FF) = 2$. An analogue of Lemma \ref{l:qc} also holds for $\ch(\FF) \neq 2$. Its proof is similar to that of Lemma \ref{l:qc}. We leave it to the reader.

\begin{lemma}\label{l:qc odd}
Let $\ch(\FF) \neq 2$ and let $q:V\to \FF$ be the generic non-degenerate quadratic form as given by~\eqref{quadform}. Then there exists a field extension $\widehat{\FF}$ such that the extension $\widehat{q}$ of $q$ to $\widehat{V}:=V\otimes\widehat{\FF}$ admits the following representation with respect to a suitable basis $\widehat{E}$ of $\widehat{V}$:
    \begin{equation}
         \label{oqe bis}
       \widehat{q}(x_1,\dots, x_N) ~ = ~
         \begin{cases}
         \displaystyle\sum_{i=1}^{N/2}x_{2i-1}x_{2i} & \mbox{if} ~ N  ~ \mbox{is even}\\[.6cm]
             \displaystyle\sum_{i=1}^{(N-1)/2}x_{2i-1}x_{2i} + \kappa_{N}x_N^2 & \mbox{if} ~ N ~ \mbox{is odd}.
                \end{cases}
            \end{equation}
Moreover, if we require that $\widehat{\FF}$ has minimal degree over $\FF$, then $\widehat{\FF}$ is a uniquely determined algebraic extension of $\FF$ of degree at most $g= 2d''_0$ where $d''_0 := \max(0, \mathrm{def}_0(q)-1)$.
\end{lemma}

The form $\widehat{q}$ of Lemma~\ref{l:qc odd} is always non-degenerate with anisotropic defect equal to $0$ or $1$ according to whether $N$ is even or odd.

 \medskip

 We shall now prove Theorem~\ref{MainTheorem bis} of Section~\ref{Main Results Sec}. Let $\widehat{\FF}$, $\widehat{V}$ and $\widehat{q}$ be as in
 Lemmas~\ref{l:qc} or~\ref{l:qc odd}. In particular, $q$ is non-degenerate. To fix ideas, suppose we have chosen $\widehat{\FF}$ so that it has minimal degree over $\FF$. Thus $\widehat{\FF}$ is uniquely determined. Accordingly, $\widehat{V}$ and $\widehat{q}$ are uniquely determined as well. Let $\cQ_k$ and $\widehat\cQ_k$ be the orthogonal $k$-Grassmannians associated to $q$ and $\widehat{q}$ respectively and
\[\varepsilon_k:\cQ_k\rightarrow\langle\varepsilon_k(\cQ_k)\rangle \subseteq \PG(V_k), \quad
\widehat{\varepsilon}_k:\widehat\cQ_k\rightarrow\langle \widehat{\varepsilon}_k(\widehat\cQ_k)\rangle\subseteq\mathrm{PG}(\widehat{V}_k)\]
be their Grassmann embeddings. So, $\cQ_k$ is a subgeometry of $\widehat\cQ_k$ and $\widehat{\varepsilon}_k$ induces $\varepsilon_k$ on $\cQ_k$. (Recall that $\PG(V_k)$ is a subgeometry of $\PG(\widehat{V}_k)$, non-full if and only if $\widehat{\FF} \supset \FF$).

Suppose $k < n$ and, when $\ch(\FF) = 2$, assume $d'_0 \leq 1$. So, $\widehat{q}$ is non-degenerate with anisotropic defect $\mathrm{def}_0(\widehat{q}) \leq 1$. Consequently, the $k$-Grassmannian $\widehat\cQ_k$ admits the Weyl embedding, say $\widehat{e}^W_k:\widehat\cQ_k\rightarrow\mathrm{PG}(\widehat{V}^W_k)$, and $\widehat{\varepsilon}_k$ is a quotient of $\widehat{e}^W_k$, as explained in Section \ref{Main Results Sec}.

Let $\widehat{\varphi}$ be the (linear) morphism from $\widehat{e}^W_k$ to $\widehat{\varepsilon}_k$ and let $e^W_k$ be the lifting of $\varepsilon_k$ to $V^W_k := \langle\widehat{\varphi}^{-1}(\varepsilon_k(\cQ_k))\rangle_{{\FF}}$ through $\widehat{\varphi}$. We know by Theorem \ref{mainlift} that $e^W_k$ is a projective embedding of $\cQ_k$ in $\mathrm{PG}(V^W_k)$ and that $\widehat{\varphi}$ induces a (linear) morphism $\varphi:V^W_k\rightarrow\langle\varepsilon_k(\cQ_k)\rangle$ from $e^W_k$ to $\varepsilon_k$. As $\widehat{q}$ is uniquely determined, in view of the hypotheses made on $\widehat{\FF}$, the embedding $e^W_k$ is uniquely determined as well.

\begin{definition}\label{Weyl-like}
The embedding $e^W_k:\cQ_k\rightarrow\mathrm{PG}(V^W_k)$ is the \emph{Weyl-like embedding} of $\cQ_k$.
\end{definition}

As $e^W_k$ is a projective embedding, $\dim(e^W_k) = \dim(V^W_k)$ by
Property~(E\ref{eE2}) of projective embeddings. As noticed in
Section~\ref{Main Results Sec}, the morphism $\widehat{\varphi}$ is an isomorphism when either $\ch(\FF) \not= 2$ or $k = 1$. In this case $e^W_k \cong \varepsilon_k$ (Remark~\ref{nova}). If $d_0 = \mathrm{def}_0(q) \leq 1$ then $\widehat{q} = q$. In this case $\FF=\widehat{\FF}$ and $e^W_k = \widehat{e}^W_k$. Accordingly, $\dim(e^W_k) = \dim(V^W_k) = {N\choose k}$.

All parts of Theorem \ref{MainTheorem bis} have been proved. It remains to estimate $\dim(e^W_k)$ (namely $\dim(V^W_k)$) when $\ch(\FF) = 2$, $k > 1$ and $d_0 > 1$. Recall that we have assumed the degree $|\widehat{\FF}:\FF|$ to be minimal and $d'_0 \leq 1$. Hence $|\widehat{\FF}:\FF| \leq 2m$, by Lemma \ref{l:qc}.

\begin{theorem}
Suppose $\ch(\FF) = 2$. Assume the hypotheses $k < n$ and $d'_0 \leq 1$, let $k > 1$ and $d_0 > 1$. Then, with $g := |\widehat{\FF}:\FF|$, we have $2\leq g\leq 2m$ and
\begin{equation}\label{ultima0}
{N\choose k}-{N\choose k-2}\leq \dim(e^W_k) \leq {N\choose k}+{N\choose k-2}(g-1).
\end{equation}
\end{theorem}
\begin{proof} By construction, $\varphi$ is the restriction of $\widehat{\varphi}$ to $V^W_k$. Hence $\ker(\varphi)\subseteq \ker(\widehat{\varphi})\cap V^W_k$. Moreover, $\ker(\varphi)$, regarded as an $\FF$-space, has dimension $\dim_\FF(\ker(\widehat{\varphi})) = g\cdot\dim_{\widehat{\FF}}(\ker(\widehat{\varphi}))$. So,
\begin{equation}\label{ultima1}
\dim(\ker(\varphi))\leq\dim_{\FF}(\ker(\widehat{\varphi}))= g\cdot\dim_{\widehat{\FF}}(\ker(\widehat{\varphi})).
\end{equation}
However
\begin{equation}\label{ultima2}
\dim_{\widehat{\FF}}(\ker(\widehat{\varphi})) = {N\choose{k-2}}
\end{equation}
since $\dim(\widehat{V}^W_k) = {N\choose k}$ while $\widehat{\varphi}(\widehat{V}^K_k) = \langle \widehat{\varepsilon}_k(\widehat\cQ_k)\rangle$ and $\dim_{\widehat{\FF}}(\langle \widehat{\varepsilon}_k(\widehat\cQ_k)\rangle) = {N\choose k}-{N\choose{k-2}}$ (Theorem~\ref{Main Theorem}, Part~\ref{pt3}). By combining~\eqref{ultima1}
with~\eqref{ultima2} we get
\begin{equation}\label{ultima3}
\dim(\ker(\varphi))\leq g\cdot{N\choose{k-2}}.
\end{equation}
However $\varphi$ maps $V^W_k$ onto $\langle\varepsilon_k(\cQ_k)\rangle$ and $\dim(\langle\varepsilon_k(\cQ_k)\rangle) = {N\choose k}-{N\choose{k-2}}$ by Theorem~\ref{Main Theorem}, Part~\ref{pt3}. By comparing these facts
with~\eqref{ultima3} we obtain
\begin{multline*}
\dim(V^W_k) = \dim(\langle\varepsilon_k(\cQ_k)\rangle) + \dim(\ker(\varphi)) \leq \\
\leq {N\choose k}-{N\choose{k-2}} + g\cdot{N\choose{k-2}} = {N\choose k} + {N\choose{k-2}}\cdot(g-1)
\end{multline*}
which yields the right hand inequality of~\eqref{ultima0}. Similarly, from
\[\dim(V^W_k) = \dim(\langle\varepsilon_k(\cQ_k)\rangle) + \dim(\ker(\varphi)) \geq {N\choose k}-{N\choose{k-2}} \]
we obtain the left part of~\eqref{ultima0}.
\end{proof}

\begin{remark}
  Since the Grassmann embedding $\varepsilon_k$ is \emph{transparent}
  \cite{ILP18} and $\varepsilon_k$ is a quotient of $e^W_k$, the latter is also transparent, i.e. the $e^W_k$-preimage of every projective line contained in $e^W_k(\cQ_k)$ is a line of the geometry $\cQ_k$.
\end{remark}

\subsection{Comments on the case $k = n$}
\label{Wg k=n}
We have not considered the case $k = n$ in Section \ref{Main Results Sec}. When $d_0 = 0$ no Weyl embedding can be defined for $\cQ_n$. Let $d_0 = 1$. Then $\cQ_n$ admits two interesting Weyl embeddings. One of them is the so-called \emph{spin embedding},  say it $e_n^{\mathrm{spin}}$.
It is projective, $2^n$-dimensional  and lives in the Weyl module for $\mathrm{O}(N,\FF)$ associated the $n$-th fundamental weight $\lambda_n$ of the root system of type $B_n$. 
The second Weyl embedding, say it $e^W_n$, is ${N\choose n}$-dimensional and lives in the Weyl module associated to the weight $2\lambda_n$. This embedding is veronesean and is closely related with $\varepsilon_n$. In fact, when $\ch(\FF) \neq 2$ then $e^W_n = \varepsilon_n$ while, if $\ch(\FF) = 2$, then $\varepsilon_n$ is a proper quotient of $e^W_n$
(see~\cite{IP13} and~\cite{IP13bis}). We call $e^W_n$ the \emph{canonical veronesean embedding} of $\cQ_n$. 

An interesting relation exists between $e^{\mathrm{spin}}_n$ and $e^W_n$. Explicitly, let $\nu$ be the veronesean embedding of the codomain $\PG(2^n-1,\FF)$ of $e^{\mathrm{spin}}_n$ in $\PG({{2^n+1}\choose 2}-1, \FF)$, which bijectively maps $\PG(2^n-1,\FF)$ onto the standard veronesean variety of $\PG({{2^n+1}\choose 2}-1, \FF)$. Then $e^W_n = \nu\cdot e^{\mathrm{spin}}_n$ (see e.g.~\cite{IP13},~\cite{IP13bis}). In particular, when $e^W_n = \varepsilon_n$ (equivalently $\ch(\FF) \neq 2$), then $\varepsilon_n = \nu\cdot  e^\mathrm{spin}_n$.

So far for the case $d_0 = 1$. We shall now consider the case $d_0 > 1$, but  we firstly state some terminology which will facilitate the discussion of this case. We say that an injective mapping $e$ from the point-set $P$ of a point-line geometry $\Gamma$ into the point-set of a projective geometry $\Sigma$ is a \emph{quadratic embedding} of \emph{rank} $r \geq 2$ if $e$ maps the lines of $\Gamma$ onto non-degenerate quadrics of Witt index $1$ spanning 
$r$-dimensional subspaces of $\Sigma$ (and in addition $e(P)$ spans $\Sigma$). Thus, the veronesean embeddings as defined in~\cite{IP13},~\cite{IP13bis}
and~\cite{Pas13} are just the quadratic embeddings of rank $2$. Morphisms of quadratic embeddings can be defined just as for veronesean embeddings (see~\cite{IP13bis}, \cite[Section 2.2]{Pas13}). 

%We can now turn to the case $d_0 > 1$.
We can now assume $d_0 > 1$ with $d'_0 = 1$ when $\ch(\FF) = 2$. As noticed in Section~\ref{Survey}, the Grassmann embedding $\varepsilon_n$ of $\cQ_n$ is quadratic of rank $d_0+1$.
We can still consider the extension $\widehat{q}$ of $q$ as in Lemmas~\ref{l:qc} and~\ref{l:qc odd}. The form $\widehat{q}$ is non-degenerate with Witt index $\widehat{n} = n+\lfloor d_0/2\rfloor$, where $\lfloor . \rfloor$ stands for integral part, and anisotropic defect $\widehat{d}_0$ equal to $0$ or $1$ according to whether $d_0$ is even or odd. In any case $\widehat{n} > n$. So, the Weyl embedding $\widehat{e}^W_n$ of $\widehat{\cQ}_n$ is projective, as well as the Grassmann embedding $\widehat{\varepsilon}_n$. The $n$-Grassmannian $\cQ_n$ is not a subgeometry of $\widehat{\cQ}_n$; nevertheless all of its points are points of $\widehat{\cQ}_n$. So, Lemma \ref{uniq} still allows to lift $\varepsilon_n$ to an embedding $e^W_n:\cQ_n\rightarrow \PG(V^W_n)$ through the morphism $\widehat{\varphi}:\widehat{e}^W_n\rightarrow\widehat{\varepsilon}_n$, for a suitable $\FF$-subspace of $\widehat{V}_n$. However $e^W_n$ cannot be projective. More explicitly, since $\widehat{\varphi}$ is linear, the lifting map behaves linearly, as shown in the proof of Theorem~\ref{mainlift}. This implies that the embedding $e^W_n$ is quadratic with the same rank $d_0+1$ as $\varepsilon_n$. Moreover $\widehat{\varphi}$ maps $V^W_n\subset \widehat{V}_n$ onto $V_n$ and induces a morphism $\varphi:e^W_n\rightarrow\varepsilon_n$ (an isomorphism if $\widehat{\varphi}$ is an isomorphism). We leave the proofs of these claims to the reader. We call $e^W_n$ the \emph{Weyl-like} embedding of $\cQ_n$. 

The case $d_0 = 2$ is particularly interesting. In this case $\widehat{\cQ}_n = \cQ_n(n+1, 0, 0;\widehat{\FF})$ and $\widehat{\FF}$, chosen of minimal degree over $\FF$, is a separable quadratic extension of $\FF$. Let $\widetilde{\cQ}_n := \cQ_n(n,1, 0;\widehat{\FF})$. Then $\widetilde{\cQ}_n$ is contained in $\widehat{\cQ}_n$ (but not as a subgeometry) while $\cQ_n$ is a full subgeometry of $\widetilde{\cQ}_n$ (see e.g.~\cite{CS01}; also~\cite[\S 3.4]{ILP18}). The geometry $\widetilde{\cQ}_n$ admits both the spin embedding $\widetilde{e}_n^{\mathrm{spin}}:\widetilde{\cQ}_n\rightarrow\PG(2^n-1,\widehat{\FF})$ and the canonical veronesean embedding $\widetilde{e}_n^W:\widetilde{\cQ}_n\rightarrow\PG({{2^n+1}\choose 2}-1,\widehat{\FF})$. The spin embedding $\widetilde{e}_n^{\mathrm{spin}}$ induces a $2^n$-dimensional projective embedding on $\cQ_n$
(see e.g.~\cite{CS01}; \cite[\S 4.2]{ILP18}). We call it the \emph{spin-like} embedding of $\cQ_n$ and use the symbol $e^{\mathrm{spin}}_n$ for it too. 
\begin{conjecture}
We conjecture that the Weyl embedding $\widehat{e}^W_n$ of $\widehat{\cQ}_n$ induces on $\widetilde{\cQ}_n$ its veronesean embedding 
$\widetilde{e}_n^W$, the codomain $V^W_k$ of the Weyl-like embedding $e^W_n$ of $\cQ_n$ is the canonical Baer subgeometry of $\PG({{2^n+1}\choose 2}-1,\widehat{\FF})$ defined over $\FF$ and $\widetilde{e}^W_n$ induces $e^W_n$ on $\cQ_n$. If so, then $e^W_n = \nu\cdot e^{\mathrm{spin}}_n$, where $\nu$ is the canonical veronesean embedding of $\PG(2^n-1,\FF)$ in $\PG({{2^n+1}\choose 2}-1,\FF)$. In particular, if $e^W_n = \varepsilon_n$ (as when $\ch(\FF)\neq 2$) then $\varepsilon_n = \nu\cdot e^{\mathrm{spin}}_n$.
\end{conjecture}

% \vskip.2cm\noindent
% \begin{minipage}[t]{\textwidth}
% Authors' addresses:
% \vskip.2cm\noindent\nobreak
% \centerline{
% \begin{minipage}[t]{7cm}
% Ilaria Cardinali and Antonio Pasini (retired) \\
% Department of Information Engineering and Mathematics\\
% University of Siena\\
% Via Roma 56, I-53100, Siena, Italy\\
% ilaria.cardinali@unisi.it \\
% antonio.pasini@unisi.it \\
% \end{minipage}\hfill
% \begin{minipage}[t]{6cm}
% Luca Giuzzi\\
% D.I.C.A.T.A.M. \\ Section of Mathematics \\
% Universit\`a di Brescia\\
% Via Branze 53, I-25123, Brescia, Italy \\
% luca.giuzzi@unibs.it
% \end{minipage}}
% \end{minipage}

\end{document}